%% file: codim1OBembedv2.tex
\documentclass[12pt]{amsart}
\usepackage{amssymb}
\allowdisplaybreaks

\usepackage{color}
\usepackage{graphics}
\usepackage{amsmath,amscd}
\usepackage{amsthm}
\usepackage[all,cmtip]{xy}
\usepackage{tikz}
\usepackage{graphicx}
\usepackage{epstopdf}
\usepackage{mathtools}
\usepackage{amssymb}
\usepackage{amsbsy}
\usepackage{bm}
\usepackage{tikz-cd}
\usepackage{enumerate}
\usepackage{enumitem}
\usepackage{float}
\usepackage[inkscapearea=page]{svg}
\theoremstyle{definition}
\newtheorem{exmp}{Example}[section]
\usepackage[colorlinks=green,linkcolor=blue,citecolor=red,urlcolor=red]{hyperref}

\newcommand{\Z}{\mathbb{Z}}

\newcommand{\Q}{\mathbb{Q}}

\newcommand{\R}{\mathbb{R}}

\newcommand{\D}{\mathbb{D}}

\newcommand{\s}{\mathbb{S}}
\newcommand{\tw}{\textrm{Twist}}

\newcommand{\id}{\textrm{id}}

\newcommand{\Id}{\operatorname{Id}}
\newcommand{\Diff}{\operatorname{Diff}}

\newcommand{\ob}{\textrm{OB}}

\newcommand{\so}{\mathrm{SO}}

\newtheorem{proposition}{Proposition}[section]
\newtheorem{theorem}[proposition]{Theorem}

\newtheorem{lemma}[proposition]{Lemma}

\newtheorem{corollary}[proposition]{Corollary}
\newtheorem{remark}[proposition]{Remark}

\begin{document}
	
	\title{Twist maps and codimension--$1$ spun embeddings}
	
	
	
	\author{Shital Lawande}
	\address{IAI TCG CREST Kolkata, and Ramakrishna Mission Vivekananda Education and Research Institute, Belur Math}
	\email{shital.lawande@tcgcrest.org}
	
	\author{Kuldeep Saha}
	\address{IAI TCG CREST Kolkata, and Academy of Scientific and Innovative Research, Gazhiabad}
	\email{kuldeep.saha@gmail.com, kuldeep.saha@tcgcrest.org}

	\begin{abstract}
		
		We study codimension--$1$ embeddings preserving open book structures. In particular, we prove that every closed orientable $3$-manifold admits a codimension--$1$ spun embedding in a finite connected sum of $\s^2 \times \s^2$s and $\s^2 \widetilde{\times} \s^2$s. We discuss some explicit constructions of planar open books on $3$-manifolds and their codimension--$1$ spun embeddings. To construct these embeddings, we use sphere twist maps and push maps. We also give a simple proof for non-triviality of the twist map along a non-separating $n$-sphere in the group of orientation preserving diffeomorphisms of $\s^1 \times \s^n \setminus int(\D^{n+1})$, relative to the boundary.   
		
	\end{abstract}
	
	\maketitle

	\section{Introduction}

		 Embedding of manifolds has been a fundamental topic in geometric topology. An interesting subtopic is to study embeddings that preserve certain geometric/topological structures on the manifolds. In the present article we focus on codimension--$1$ smooth embeddings of closed manifolds, that preserve open book structures. An open book decomposition of a manifold $M^m$ is a pair $(V^{m-1},h)$, such that $M^m$ is diffeomorphic to the quotient space $\mathcal{MT}(V^{m-1}, h) \cup_{id} \partial V^{m-1} \times D^2$. Here, $V^{m-1}$, called \emph{page}, is a manifold with boundary. The map $h$, called \emph{monodromy}, is a diffeomorphism of $V^{m-1}$ that restricts to identity near the boundary $\partial V$, and $\mathcal{MT}(V^{m-1}, h)$ denotes the mapping torus of $h$. We denote such an open book by $\textrm{OB}(V,h)$. It is well known that every closed, orientable, odd dimensional manifold admits open book decomposition. For even manifolds necessary and sufficient conditions are known except in dimension $4$. An embedding $f$ of $\textrm{OB}(V_1,h_1)$ in $\textrm{OB}(V_2,h_2)$ is called an \emph{spun embedding}, if $f$ restricts to a proper embedding of $V_1$ in $V_2$ and $h_2 \circ f = f \circ h_1$ (up to isotopy).

		 \noindent The problem of spun embedding is most interesting for lower codimensions of embedding. In particular, the codimension $2$ case has recieved much attention in recent times (eg. \cite{EL}, \cite{pps}, \cite{saha}, \cite{ls1}). In this article, we initiate a study of codimension--$1$ spun embeddings of closed oriented manifolds. We first prove a codimension--$1$ open book embedding result for closed oriented $3$-manifolds. By Lickorish \cite{lickorish} and Wallace \cite{wallace}, every closed oriented $3$-manifold $M$ can be obtained from $\s^3$ by $\pm1$-surgery along an $n$-component link $\mathcal{L}_M$ of unknots. Given such a pair $(M,\mathcal{L}_M)$, one can construct a planar open book decomposition of $M^3$ with page $\Sigma_{0,n+1}$ (see section \ref{planarob}). Let $M$ be a closed oriented $3$-manifold and let $\ob(\Sigma_{0,n+1}, \phi_M)$ be a planar open book decomposition of $M$. Thus, every closed orientable $3$-manifold admits a planar open book decomposition obtained from a surgery diagram.

	\begin{theorem} \label{thm0}
		
	$M = \textrm{OB}(\Sigma_{0,n+1},\phi_M)$ open book embeds in the $4$-manifold $W_{i,j} = (\#^i\s^2 \times \s^2) \# (\#^j \s^2 \widetilde{\times}\s^2)$, for some $i,j \in \mathbb{N}$, such that $i + j = n$. 
		
	\end{theorem}
	
	\noindent The numbers $i$ and $j$ in Theorem \ref{thm0} are determined by a Dehn twist presentation of the monodromy map $\phi_M$ (see section \ref{codim1spunembed}). 
	
	\noindent In section \ref{sec2}, we discuss the necessary background materials. We also give a proof of the fact that every finitely presented group can be realized by a 4-dimensional open book. This has the following interesting consequence (see Corollary \ref{cor01}).
	
	\begin{theorem} 
		Every finitely presented group is the fundamental group of a closed oriented $4$-manifold having simplicial volume zero. 
	\end{theorem}
	
	\noindent In section $3$, we construct some explicit codimension--$1$ spun embeddings of the lens spaces (see Theorem \ref{lensembed}), some small Seifert spaces and the Poincare homology $3$-sphere in section \ref{codim1spunembed}. We also give an example of codimension--$1$ spun embedding of $4$-manifold in $\s^5$ (see section \ref{spunembeds5}).

	Our main ingredients to prove the codimension--$1$ spun embedding results are \emph{twist} maps. In particular, we use the \emph{sphere twist map} and the \emph{torus twist map}, which we also call the \emph{push map}. The twist maps are an interesting family of maps that frequently occurs in the study of mapping class groups of $3$-manifolds. They also play an important role in constructing explicit descriptions of even dimensional open books (see \cite{hsueh1} and \cite{hsueh2}). In general, it is not easy to show non-triviality of a sphere twist, as it can not be detected by most of the basic invariants in algebraic topology. In section \ref{twistnontriviality}, we give a simple proof of non-triviality of the sphere twist map along a non-separating embeddded $\s^{n-1}$ in an $n$-manifold, using open book techniques. See section \ref{spheretwist} for definition of a sphere twist. Let $V^n$ be a connected orientable manifold with non-empty boundary that admits a proper embedding in $\D^{n+2}$ ($n \geq 2$). Let $\sigma_{n-1}$ denote the twist along the non-separating sphere $\{*\}\times \s^{n-1}$ in $\s^1 \times \s^{n-1}$. Then we show the following.
	
	\begin{theorem}\label{thm2}
		
		The sphere twist $\sigma_{n-1}$ is not isotopic to identity in $Diff^{+}_\partial (\s^1 \times \s^{n-1} \# V)$.
		
	\end{theorem}

    \noindent Here, $Diff^{+}_\partial(X)$ denotes the group of orientation preserving diffeomorphisms of $X$, relative to the boundary. There exists many examples of manifolds $V^n$ that admit proper embeddings in $\D^{n+2}$. For example, Cochran \cite{Coch} has given sufficient conditions for a closed $4$-manifold $X$ such that $X_0 = X \setminus int(\D^4)$ smoothly embeds in $\s^5$. Any such $X_0$ admits a proper embedding in $\D^6$. In particular, any simply simply connected $4$-manifold that embed in $\s^6$, also embeds in $\s^5$. Simple $4$-manifolds of the form $\s^1 \times M^3$, $\Sigma_{g_1} \times \Sigma_{g_2}$ also embed in $\s^5$. Also, the page of any open book decomposition on $\s^5$ gives example of such a $4$-manifold.    
    
    \noindent Let $M^n_k = \#_b^k (\s^{n-1} \times [0,1])$ and let $\sigma_i$ denote the sphere twist along the core sphere (which is separating) of the $i$th copy of $\s^{n-1} \times [0,1]$. Then $\sigma_i$ is not isotopic to identity in $\Diff_\partial(M^n_k)$. One can see this in the following way. We note that $\ob(M_k^n,\sigma_i) = \ob(\s^{n-1},\id) \# \cdots \# \ob(\s^{n-1} \times [0,1],\id) \# \ob(\s^{n-1} \times [0,1],\sigma_i) \#\ob(\s^{n-1},\id) \# \cdots \# \ob(\s^{n-1} \times [0,1],\id) = \#^{k-1}(\s^2\times\s^{n-1}) \# (\s^2 \widetilde{\times} \s^{n-1})$ (see Lemma \ref{nospinlemma}). Thus, $\ob(M_k^n,\sigma_i)$ is not spin. If $\sigma_i$ is trivial in $\Diff_\partial(M^n_k)$, then $\ob(M_k^n,\sigma_i) = \#^k (\s^2 \times \s^{n-1})$, which is spin, giving a contradiction.

    Unless stated otherwise, we always work in the smooth category.
    
    \subsection{Acknowledgement} The authors warmly thank Chun-Sheng Hsueh for discussing his work on $4$--dimensional open books and twist maps \cite{hsueh2}, which motivated this article. We also thank Professor Sukumar Das Adhikari and Professor Goutam Mukherjee for their support and encouragement during this work.

	\section{Preliminaries}	\label{sec2}

	\subsection{The sphere twist map} \label{spheretwist}
	
	We consider the action of $\so(n+1)$ on $\s^n$ ($n \geq 1$). It is well known that the inclusion map $\so(n+1)  \hookrightarrow \Diff^+(\s^n)$ induces an injective homomorphism $\pi_1(\so(n+1)) \hookrightarrow \pi_1(\Diff^+(\s^n))$. Let $\alpha$ be the generating element in $\pi_1(\so(n+1))$.  One can then define a diffeomorphism $ \sigma_n : \s^n \times [0, 1] \rightarrow \s^n \times [0, 1]$ by $(y, t) \mapsto (\alpha(t)\cdot y, t)$ for $(y, t) \in \s^n \times [0, 1]$. If $\s^n$ is smoothly embedded in an oriented $n+1$ manifold $M$, then $\sigma_n$ can be extended to a diffeomorphism $M$, called a \emph{twist map} along $\s^n$. 
	
	\noindent Note that $\sigma_1$ is the Dehn twist map along an annulus $\mathcal{A}$ and $\pi_0(\Diff^+(\mathcal{A}, \partial \mathcal{A}))$ is the infinite cyclic group generated by $\sigma_1$. For $n \geq 2$, $\sigma_n$ generates the group $\pi_0(\Diff^+_\partial(\s^n \times [0,1])) = \Z_2$ relative to boundary, where, $\Diff^+_\partial(M)$ denotes the group of orientation preserving diffeomorphisms of a manifold $M$ that restricts to identity near $\partial M$. 
	
	\begin{lemma}\label{nospinlemma}
		$\textrm{OB}(\s^n \times [0,1], \id) = \s^2 \times \s^n$ and $\textrm{OB}(\s^n \times [0,1], \sigma_n) = \s^2 \widetilde{\times} \s^n$.
	\end{lemma} 
	
	\begin{proof}
		
	By definition, $\textrm{OB}(\s^n \times [0,1], \id) = \s^n \times [0,1] \times \s^1 \cup_{id, \partial} \s^n \times \{0,1\} \times \D^2 = \D^2 \times \s^n \times \{0\} \cup_{id, \partial} \D^2 \times \s^n \times \{1\} = \s^2 \times \s^n$. If we cut open $\s^2 \times \s^n$ along a page of its open book decomposition, we get $V^{n+2} = \s^n \times \D^2$. The complement of $V \subset \s^2 \times \s^n$ is also diffeomorphic to $\s^n \times \D^2$, where $\{p\} \times \D^2$ is identified with the suspension of $\{p\} \times [0,1]$ for every $p \in \s^n$. 
	
	\noindent  We note that $\s^2 \widetilde{\times } \s^n = \mathcal{MT}(\s^n \times [0,1], \sigma_n) \cup_{id, \partial} \s^n \times \{0,1\} \times \D^2$ is obtained from $\s^2 \times \s^n$ by cutting open along a page and gluing back a copy of $\s^n \times \D^2$ via the map $\sigma_n : (y, t) \mapsto (\alpha(t)\cdot y, t)$. Here, $\alpha(0) = \alpha(1) = \Id_{n+1}$. Thus, $\textrm{OB}(\s^n \times [0,1], \sigma_n) = \s^2 \widetilde{\times} \s^n$.
	
	\end{proof}

\subsection{The $2$d sphere twist} \label{2dtwist} Let $M^3$ be an oriented $3$-manifold (possibly with boundary) and let $S$ be an $2$-sphere embedded in $M$. Let $\sigma_S$ denote the sphere twist on $M$ along $S$. The isotopy class of $\sigma_S$ only depends on the isotopy class of $S$. In fact, Laudenbach \cite{laud1} \cite{laud2} showed that if $S$ and $\hat{S}$ are homotopic $2$-spheres in $M$ that are non-nullhomotopic, then $S$ and $\hat{S}S$ are isotopic. Thus, $\sigma_S$ only depends on the homotopy class of $S$.

\begin{figure}[htbp] 
	
	\centering
	\def\svgwidth{5cm}
	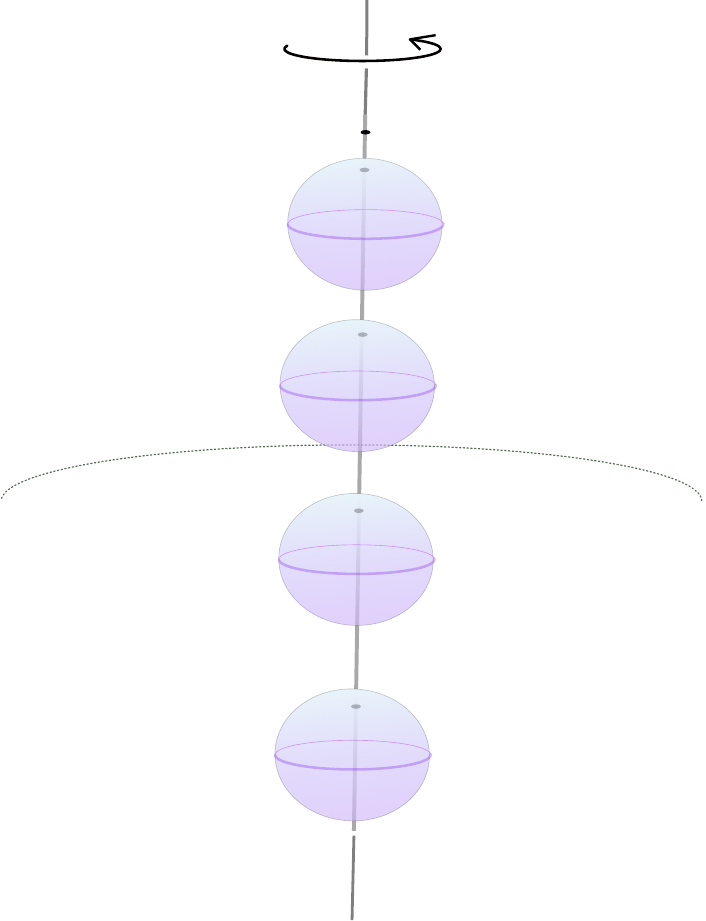
	\caption{}
	\label{figstwist}
	
\end{figure}

\noindent Let $\tw(M^3)$ denote the \emph{twist} subgroup of $\Diff_\partial(M)$ generated by sphere twists along embedded spheres (up to isotopy) in $M$. In general, it is not easy to show that $\sigma_S \in \tw(M)$ is non-trivial for some $S$. The reason being that $\sigma_S$ fixes the homotopy class of any loop and surface in $M$ (see page $2$ of \cite{bbp}). Moreover, for $f \in \Diff_\partial(M)$, $f \circ \sigma_S \circ f^{-1} = \sigma_{f(S)}$. Let $\hat{S}$ be another embedded $2$-sphere in $M$. Since $\sigma_S(\hat{S})$  and $\hat{S}$ belong to the same homotopy class, $\sigma_S \circ \sigma_{\hat{S}} = \sigma_{\hat{S}} \circ \sigma_S$. Thus, $\tw(M)$ is abelian. Let $D_n^3$ denote the boundary connected sum of $n$-copies of $\s^2 \times [0,1]$. The following was observed by Hatcher and Wahl (see section $2$ of \cite{hw}).

\begin{proposition}[\cite{hw}]\label{twistprop}
For any embedding of $D^3_n$ in a $3$-manifold $M$, the composition of the $n+1$ twists along the boundary spheres of $D^3_n$ is isotopic to the identity. 
	
\end{proposition}

One can construct an isotopy from the identity map of $D^3_n$ to the composition of twists along the boundary spheres in the following way. One can also think of $D_n^3$ as a $3$-ball with $n$ disjoint sub-balls removed from its interior.  Align the sub-balls along an axis of a standard $\D^3$ as in Figure \ref{figstwist} and then rotate each point $x$ in the region between the inner and outer boundary spheres of $D^3_n$ by an angle $t\theta(x)$ about the axis.  Here $t \in [0,1]$ is the isotopy parameter and $\theta(x)$ goes from $0$ to $2\pi$ as $x$ varies across an $\epsilon$-neighborhood of the boundary spheres, with $\theta(x) = 0$ on the spheres and $\theta(x) = 2\pi$ outside the neighborhood.  Thus, if $S_1,\dots , S_{n+1}$ denotes boundary spheres of $D^3_n$, then $\sigma_{S_1} \circ \sigma_{S_2} \circ \cdots \circ \sigma_{S_{n+1}} \simeq \id$.  Let $S_{n+1}$ denote the outermost boundary sphere.

\subsection{The push map} Let us consider a planar surface with $3$ boundary components $\Sigma_{0,3}$. Let $a$ and $c$ denote the inner components and let $b$ denote the outer component (see Figure \ref{figS03}). If we \emph{push} the component $c$ around $a$ and come back to the initial position we get a diffeomorphism of $\Sigma_{0,3}$, relative to boundary, called the \emph{push map} along $a$. In terms of Dehn twists, the push map is isotopic to $\tau_a \tau^{-1}_b$. The \emph{push map} can also be defined for the manifold $\s^n \times [0,1] \#_b \s^1 \times \D^n$. In this case, we need to push the boundary component diffeomorphic to $\s^n$, along the curve $\gamma = \s^1 \times \{p\}$ for some point $p \in int \D^n$. Here, by \emph{pushing} along $\gamma$, we are essentially defining an element of $\pi_1 \textrm{Emb}(\s^n, \s^n \times [0,1] \#_b \s^1 \times \D^n)$. This loop of embeddings then defines an element $\mathcal{P}_{n+1} \in \Diff^+_\partial(\s^n \times [0,1] \#_b \s^1 \times \D^n)$. The push map is a special case of a more general class of diffeomorphisms on the manifold $\mathcal{B}^n_{i,j} = \s^i \times \D^{n-i} \#_b \s^j \times \D^{n-j}$, called the \emph{barbell map} and it was introduced by Budney and Gabai \cite{BG}.

\noindent In fact, the push map can be defined for an arbitrary simple closed curve along which a sphere is pushed around. Let $M^3$ be an oriented 3-manifold with a sphere boundary component. A \emph{push map} along a simple closed curve $\alpha \in int(M)$ is an orientation preserving diffeomorphism of $M$, relative to boundary, which is obtained by pushing the sphere boundary component along the curve $\alpha$ and bringing it back to its initial position, while keeping the other boundary components pointwise fixed. For example, let $M$ be the manifold $\mathcal{B}_{1,1} = \s^2 \times [0,1] \#_b \s^1 \times \D^2$ and  $\alpha$ be a simple closed curve in $\mathcal{B}_{1,1}$, as in Figure \ref{figaro1p}. Pushing the sphere boundary component along $\alpha$, defines an element of $\pi_1 \textrm{Emb}(\s^2, \s^2 \times [0,1] \#_b \s^1 \times \D^2)$. This loop of embeddings defines an element $\rho_\alpha \in \Diff^+_\partial(\s^2 \times [0,1] \#_b \s^1 \times \D^2)$. In general, given a sphere boundary component $S$ of $M$, we denote the push map that pushes $S$ along $\alpha$, by $\rho_{S,\alpha}$. If $\alpha_0$ is a simple closed curve isotopic to $\s^1 \times \{p\} \subset \s^1 \times \D^2 \subset \mathcal{B}_{1,1}$, then $\ob(\mathcal{B}_{1,1}, \rho_{\alpha_0}) = \s^4$ (see \cite{hsueh1} for a proof). 

\begin{figure}[htbp] 
	
	\centering
	\def\svgwidth{16cm}
	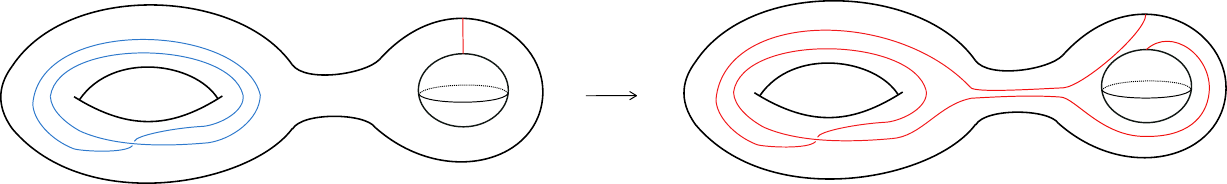
	\caption{The closed curve obtained as the union of an arc joining the two boundary components of $\mathcal{B}_{1,1}$ and its image under $\rho_\alpha$, is homotopic to $\alpha$.}
	\label{figaro1p}
	
\end{figure}

\noindent We note that $\rho_{S,\alpha}$ corresponds to a loop in $\textrm{Emb}(\s^2, M)$. This gives a proper embedding $f_{S,\alpha} : \s^1 \times \D^2 \setminus \mathcal{S}(\{pt.\} \times \D^2) \rightarrow M$, where $\mathcal{S}(\{pt.\} \times \D^2)$ denotes the suspension of the disk $\{pt.\} \times \D^2$. We note that $\rho_{S,\alpha}$ is supported in a neighborhood of the image of $f_{S,\alpha}$.

\subsection{Open book decomposition} An open book decomposition of a closed $n$-manifold $M$ consists of a codimension $2$ closed submanifold $B$ and a fibration map $\pi : M \setminus B \rightarrow \s^1$, such that in a tubular neighborhood of $B \subset M$, the restriction map $\pi : B \times (\D^2 \setminus \{0\}) \rightarrow \s^1$ is given by $(b,r,\theta) \mapsto \theta$. The fibration $\pi$ determines a unique fiber manifold $N^{n-1}$ whose boundary is $B$. The closure $\bar{N}$ is called the \emph{page} and $B$ is called the \emph{binding}. The monodromy of the fibration map determines a diffeomorphism $\phi$ of $\bar{N}$ such that $\phi$ is identity near a collar neighborhood of boundary $\partial \bar{N}$. In particular, $M = \mathcal{MT}(\bar{N}, \phi) \cup _{id, \partial} \partial \bar{N} \times \D^2$, where $\mathcal{MT}(\bar{N}, \phi)$ denotes the mapping torus of $\phi$. We denote such an open book decomposition of $M$ by $\textrm{OB}(\bar{N},\phi)$. The map $\phi$ is called the \emph{monodromy} of the open book. It is a well-known fact that every odd dimensional closed manifold admits an open book decomposition \cite{Al} \cite{La}. The existence question of open book decomposition on a closed $4$-manifold is not resolved yet.

	\begin{exmp} 
		The \emph{trivial} open book on $\s^n$ is given by $\textrm{OB}(\D^{n-1}, id)$. 
	\end{exmp} 
	
	\begin{exmp} 
	  Consider the manifold $\mathcal{W}_n = \s^{n-1} \times [0,1] \#_b \s^1 \times \D^{n-1}$ ($n \geq 2$) obtained by taking boundary connected sum of $\s^{n-1} \times [0,1]$ and $\s^1 \times \D^{n-1}$. Recall the push map $\mathcal{P}_n$ defined on $\mathcal{W}_n$. The following is well-known (see \cite{hsueh1} and \cite{hsueh2}).
	  
	  \begin{lemma}
	  	$\ob(\mathcal{W}_n, \mathcal{P}_n) = \s^{n+1}$.
	  \end{lemma}
	  
	\end{exmp} 
	
\subsection{Fundamental group of 4D open books} A well-known fact in topology is that every finitely presented group can be realized as the fundamental group of an orientable closed $4$-manifold (Theorem $1.2.33$ in \cite{gs}). The push maps can be used to prove the following interesting fact.

\begin{theorem} \label{thmcontrv} \label{thmfund4}
	
	Given a finitely presented group $\mathcal{G}$, there exists an orientable $4$-dimensional open book $V^4 = \ob(M,\phi)$ such that $\pi_1(V) = \mathcal{G}$.   
	
\end{theorem}

\noindent Theorem \ref{thmcontrv} is known essentially due to Kegel-- Schmaschke \cite{KegSc}. However, since the authors could not find a proof in the literature, we give here a proof of Theorem \ref{thmcontrv}.

\begin{proof}[Proof of Theorem \ref{thmfund4}]
	
	Let $\mathcal{G} = \langle x_1,x_2,\dots,x_g | r_1,r_2,\dots,r_k \rangle$. We define $M_0 = \#^g_b(\s^1 \times \D^2) \#^k_b (\s^2 \times [0,1])$. Note that $\partial M_0 = \Sigma_g \sqcup (\sqcup^k \s^2)$. Let $S_1, S_2, \dots, S_k$ be the sphere boundary components. Let $a_1, a_2, \dots, a_g$ be the generating curves of $\pi_1(M_0)$ and $\theta_j$ be an arc joining $\Sigma_g$ to $S_j$, as shown in Figure \ref{figpagep}. 
	
	\begin{figure}[htbp] 
		
		\centering
		\def\svgwidth{10cm}
		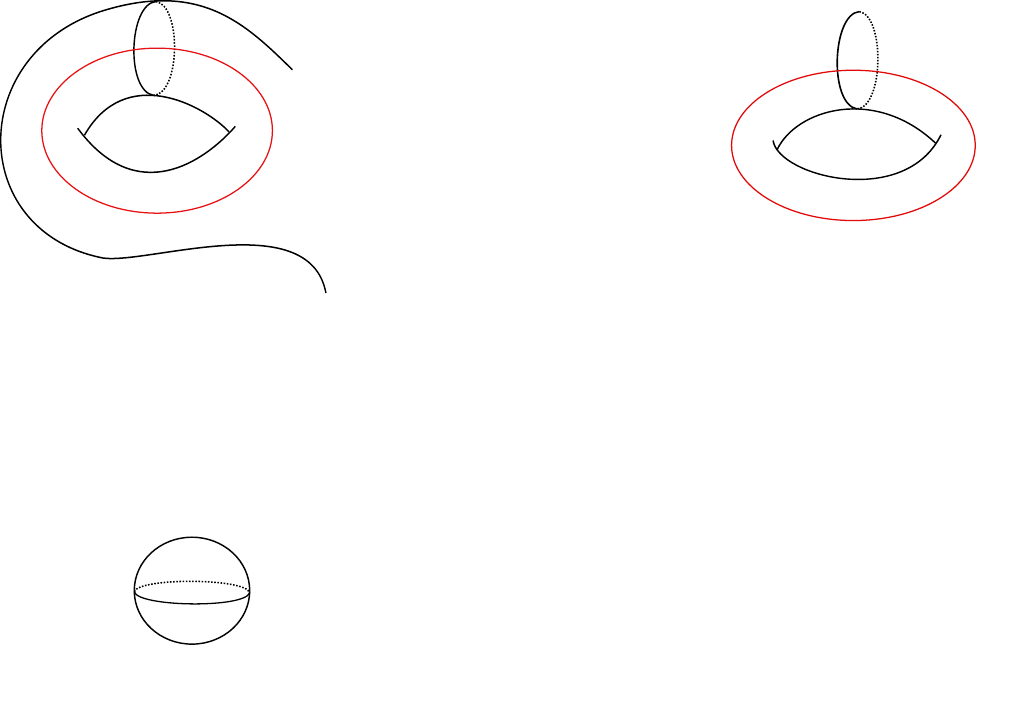
		\caption{$M_0$}
		\label{figpagep}
		
	\end{figure}
	
	\noindent	Let $r_j = a^{\delta_1}_{i_1} a^{\delta_2}_{i_2}\cdots a^{\delta_{l(j)}}_{i_{l(j)}}$ for $1 \leq j \leq k$, and let $\gamma_j$ be a simple closed curve in the interior of $M_0$ homotopic to $r_j$. We define an element $\phi \in \Diff^+_\partial(M_0)$ by $\phi = \rho_{S_1,\gamma_1}\circ \rho_{S_2,\gamma_2}\circ \cdots \circ \rho_{S_k,\gamma_k}$. Define $V = \ob(M_0,\phi)$. We claim that $\pi_1(V) = \mathcal{G}$. To see this, we first observe that $\pi_1(V) = \langle a_1, a_2, \dots , a_g | a^{-1}_1\phi_*(a_1), a^{-1}_2\phi_*(a_2) , \dots ,a^{-1}_g\phi_*(a_g) , \bar{\theta_1} \phi_*(\theta_1), \bar{\theta_2}\phi_*(\theta_2)\cdots \bar{\theta_k}\phi_*(\theta_k) \rangle$. Note that the union of all such $x_i$s form a wedge of $g$ circles with a wedge point $p$. A regular neighborhood $\mathcal{N}_g$ of this wedge is homeomorphic to a handlebody of genus $g$. Since a simple closed curve $\alpha$ can be isotoped to be disjoint from $\mathcal{N}_g$, the push map $\rho_{S,\alpha}$ is supported away from $\mathcal{N}_g$. Hence, $\rho_{S,\alpha}$ restricts to the identity map on $\mathcal{N}_m$. Thus, $a^{-1}_j\phi_*(a_j)$ is the identity element for $1 \leq j \leq g$. Finally, note that $\bar{\theta_i} \phi_*(\theta_i) = \gamma_i = r_i$ in $\pi_1(V)$ for $1 \leq i \leq k$.
\end{proof}

\noindent Recently, Kastenholz \cite{tk} has proved that every $4$-manifold admitting an open book decomposition has simplicial volume zero. Thus, Theorem \ref{thm1} implies the following.

\begin{corollary} \label{cor01}
	Every finitely presented group is the fundamental group of a closed oriented $4$-manifold having simplicial volume zero. 
\end{corollary}

\begin{remark}
The proof above can be generalized in all dimensions by considering the higher dimensional push maps. In particular every finitely presented group can be realized as the fundamental group of an $n$-dimensional open book for $n \geq4$
\end{remark}
	
	\subsection{Open book embedding} We say, $\textrm{OB}(\Sigma_1,\phi_1)$ \emph{open book embeds} in $\textrm{OB}(\Sigma_2,\phi_2)$, if there exists a proper embedding $g : (\Sigma_1,\partial \Sigma_1) \rightarrow (\Sigma_2,\partial \Sigma_2)$ such that $g \circ \phi_1$ is isotopic to $\phi_2 \circ g$ (relative to boundary). Another convenient term for such embeddings is \emph{spun embedding}. 
	
	\begin{exmp}
		
		Since any two proper embeddings of a surface in $\D^5$ are isotopic (due to Whitney), every $3$-dimensional open book admits an open book embedding in $\s^6 = \ob(\D^5,\id)$. A similar statement also holds in higher dimension, as any two proper embeddings of an $n$-manifold ($n \geq 2$) are isotopic in $\D^{2n+1}$.   
		
	\end{exmp}

	\subsection{Planar open book decomposition of a $3$-manifold} \label{planarob}Alexander first showed that every orienteable closed $3$-manifold admits an open book decomposition. We now review the construction of a planar open book decomposition on a $3$-manifold due to Onaran \cite{onaran}. The construction is based on the following lemma due to Etnyre \cite{Et0}.

	\begin{lemma}[\cite{Et0}] \label{etlemma}
		
		Let $M$ be a closed oriented $3$-manifold, and let $(\Sigma,\phi)$ be an open book decomposition of $M$. 
		
		\begin{enumerate}
			
			\item[(a)] Let $K$ be a knot in $M$ such that $K = \{x\} \times [0, 1]/ \phi$ in M for
			some point $x$ in the interior of $\Sigma$. Let $D_x$ denote an open disk which is a neighborhood of $x$ on $\Sigma$ such that $\phi|_{D_x} = \id$. Then $0$-surgery along $K$ gives a new manifold with an open book decomposition having the surgery dual of $K$ as one of the binding components. The surface $\Sigma' = \Sigma \setminus \ D_x$ is the page and the map $\phi' = \phi |_{\Sigma'}$ is the monodromy of the new open book.
			
			\item[(b)] Let $K$ be a knot in $M$ sitting on a page $\Sigma$ of the open book decomposition $(\Sigma, \phi)$. Then $\pm1$-surgery along $K$ with respect to the page framing gives a new manifold with an open book $\ob(\Sigma, \phi \circ \tau_K^{\pm1})$, where $\tau_K^{\pm1}$ denote the right (\emph{or} left) handed Dehn twist.
		\end{enumerate}
		
	\end{lemma}

	\begin{figure}[htbp] 
		
		\centering
		\def\svgwidth{11cm}
		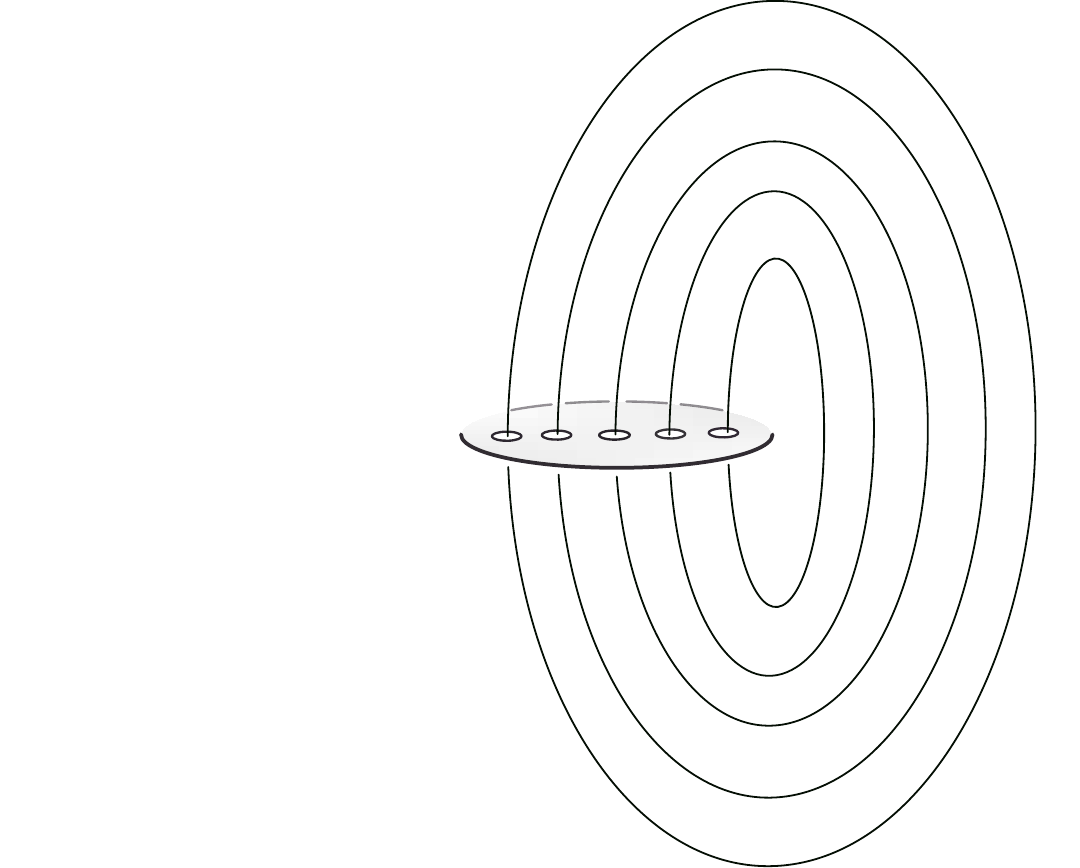
		\caption{}
		\label{figbraid}
		
	\end{figure}
	
	\noindent By Lickorish \cite{lickorish} and Wallace \cite{wallace}, every closed oriented $3$-manifold can be obtained from $\s^3$ by $\pm 1$ surgery along an $n$-component link $\mathcal{L}$ of unknots.  Moreover, $\mathcal{L}$ can be chosen so that $\mathcal{L}$ is braided about an unknot $U$ in $\s^3$ and each component of $\mathcal{L}$ links $U$ exactly once. Thus, $\mathcal{L}$ can be represented as a \emph{pure} braid with $n$ strands. The pure braid group on $n$-strands is generated by the braids $A_{i,j}$ ($1 \leq i < j \leq n$). See Figure \ref{figbraid}. We can starts with the trivial open book of $\s^3$ with disk page such that this braided link intersects a disk page in $n$ disjoint points. Our candidate planar page is obtained by deleting disjoint open disks around these $n$ interesect points. To apply statement $(a)$ of Lemma \ref{etlemma}, we need to undo the twist $A_{i,j}$. For this we apply Dehn twist on a parallel page, along a simple closed curve homologous to the sum of the $i$th and $j$th interior boundary components. The Dehn twist is positive (\emph{or} negative) if the linking number between the $i$th and $j$th strands is $-1$(\emph{or} $1$). This is equivalent to a \emph{Rolfsen twist} on the surgered $3$-manifold and thus does not change the diffeomorphism type. We also need that the surgery coefficients along each strand be zero. For example, if the $i$th strand has surgery coefficient $m_i$ ($> 0$), then we consider $m_i$ disjoint parallel copies of the page and apply a negative Dehn twist on each copy. This is equivalent to doing the  \emph{blow down} operations on the surgered $3$-manifold along an unlink with $m_i$ components. Therefore, the diffeomorphism type is again unchanged. Thus, we get an open book decomposition of the surgered $3$-manifold with page $\Sigma_{0,n+1}$ and monodromy $\phi$, where $\phi$ is a composition of Dehn twists along the boundary parallel curves and simple closed curves homologous to the sum of the $i$th and $j$th interior boundary components.

	\begin{figure}[htbp] 
	
	\centering
	\def\svgwidth{12cm}
	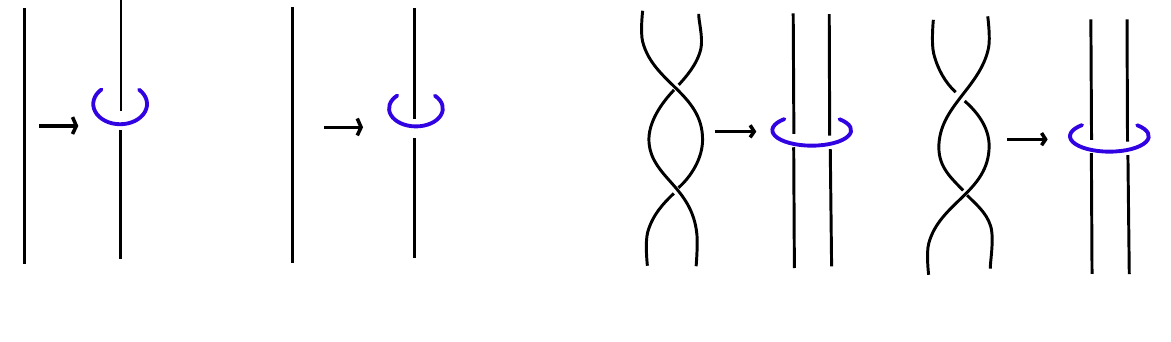
	\caption{}
	\label{figblowup}
	
\end{figure}

\noindent Similar methods can be used to create planar open book decompositions from other surgery diagrams also. We shall discuss some explicit constructions of planar open books in section \ref{codim1spunembed}. The following was proved by Onaran \cite{onaran}.

	\begin{theorem}[Onaran \cite{onaran}] \label{thmonaran}
		
	Let $M$ be a closed oriented $3$-manifold obtained by $\pm1$ surgery on $\mathcal{L} \subset \s^3$. Then, $M$ has a planar open book decomposition whose page is a disk with $n$ punctures and whose monodromy is obtained from a pure braid presentation of $\mathcal{L}$ and its surgery coefficients.
		
	\end{theorem}

	\section{Codimension--1 spun embedding} \label{codim1spunembed} 
	
		It is well known that the inclusion $\so(2) \hookrightarrow \so(3) \hookrightarrow \dots \hookrightarrow \so(k)$ ($k \geq 3$) induces an epimorphism from $\pi_1(\so(2)) = \Z \rightarrow \Z_2 = \pi_1(\so(n))$. If we consider the standard embedding of $\s^n$ in $\s^{n+1}$ as the equator $n$-sphere, then the restriction of $\sigma_{n+1}$ on $\s^n \times [0,1]$ is $\sigma_n$. In particular, we get the following.

	\begin{theorem}\label{thm1}
		
		For $n \geq 1$, $\ob(\s^n \times [0,1], \sigma^i_n)$ open book embeds in $\ob(\s^{n+1} \times [0,1], \sigma^i_{n+1})$.

	\end{theorem}

	\noindent Let $M = \ob(\Sigma_{0,n+1}, \phi_M)$ be a planar open book decomposition and say $\phi_M$ has a word presentation in terms of Dehn twists along a finite set of simple closed curves $\mathcal{C}$ on $\Sigma_{0,n+1}$. Let $\delta_i$ ($1 \leq i \leq n$) denote the simple closed curves parallel to the boundary components of $\Sigma_{0,n+1}$. see Figure \ref{figplanar}.
	
	\begin{figure}[htbp] 
		
		\centering
		\def\svgwidth{10cm}
		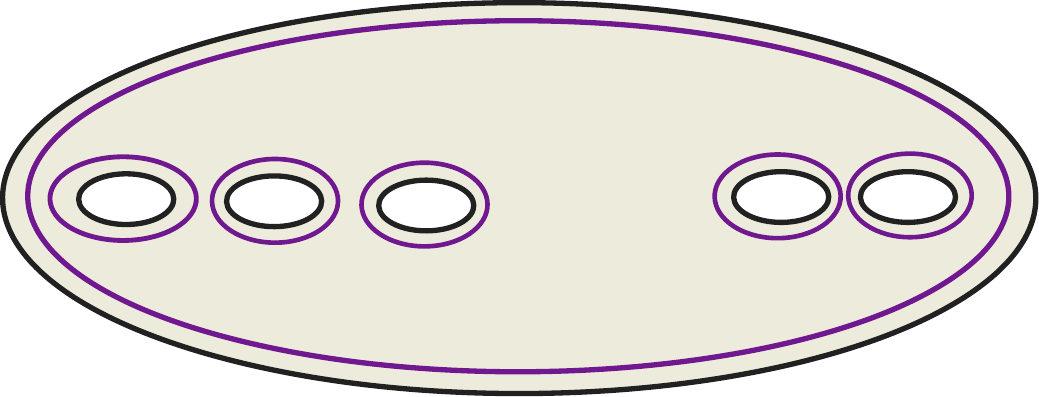
		\caption{}
		\label{figplanar}
		
	\end{figure}
	
	\begin{proof}[Proof of Theorem \ref{thm0}]
		
	Consider $\s^1$ embedded in $\s^2$ as the great circle. This induces a canonical embedding of $\s^1 \times [0,1]$ in $\s^2 \times [0,1]$. Let $V_n$ denote the boundary connected sum of $n$ copies of $\s^2 \times [0,1]$. We can take the boundary connected sum so that it induces an embedding $\iota$ of $\Sigma_{0,n+1}$ as the boundary connected sum of $n$ copies of $\s^1 \times [0,1]$. See Figure \ref{fig3demb}.

	\begin{figure}[htbp] 
		
		\centering
		\def\svgwidth{10.4cm}
		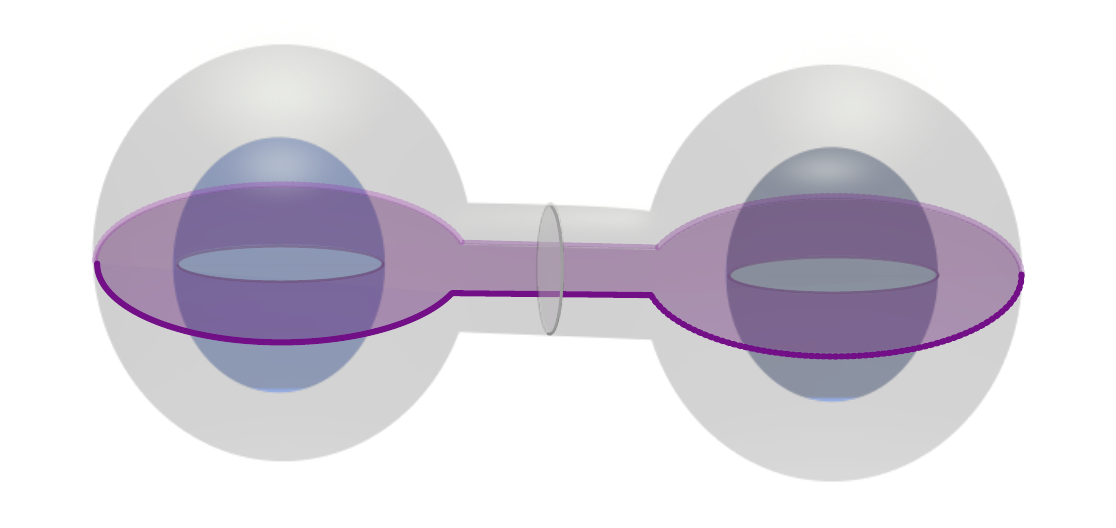
		\caption{A proper embedding of $\Sigma_{0,3}$ in $(\s^2 \times [0,1]) \#_b (\s^2 \times [0,1])$.}
		\label{fig3demb}
		
	\end{figure}

	\noindent The curve $\delta_i$ can be identified with the core circle of the $i$-th copy of $\s^1 \times [0,1]$ in $\Sigma_{0,n+1}$ for $i=1,2,...,n$. Let $S_i$ denote the core $2$-sphere of the corresponding $\s^2 \times [0,1]$ in $V_n$. Let $\gamma_{ij}$ be a simple closed curve on $\Sigma_{0,n+1}$ such that $[\gamma_{ij}] = [\delta_i] + [\delta_j]$, where $[c]$ denotes the homology class. Let $S_{ij}$ be an embedded $2$-sphere in $V_n$ obtained by tubing $S_i$ and $S_j$ such that $\gamma_{ij}$ is a great circle of $S_{ij}$ (see Figure \ref{fig3demb}). Note that $[S_{ij}] = [S_i] + [S_j]$. As observed before, the $2$-dimensional twist map along $S_{ij}$, $\sigma_{S_{ij}}$ induces a Dehn twist along $\gamma_{ij}$. By Proposition \ref{twistprop}, $\sigma_{S_{ij}} = \sigma_{S_j}^{-1} \circ \sigma_{S_i}^{-1} = \sigma_{S_i}^{-1} \circ \sigma_{S_j}^{-1}$. Similarly, for a simple closed curve $\gamma_{i_1,i_2,\dots,i_k}$ homologous to $[\delta_{i_1}] + [\delta_{i_2}] + \cdots + [\delta_{i_k}]$, we can find a $2$-sphere $S_{i_1i_2\dots i_k}$ such that $\sigma_{S_{i_1\dots i_k}}$ induces the Dehn twist on $\Sigma_{0,n+1}$ along $\gamma_{i_1,i_2,\dots,i_k}$, and $\sigma_{S_{i_1\dots i_k}}$ is generated by the $\sigma_{S_{i_j}}$s ($1 \leq j \leq k$). Now, $\phi_M = \prod_{\gamma_1,\dots,\gamma_m \in \mathcal{C}} \tau^{\alpha_1}_{\gamma_1} \circ \cdots \circ \tau^{\alpha_m}_{\gamma_m}$, for some integers $m, \alpha_1,\dots \alpha_m$. For each $\gamma_i$ we can find a $2$-sphere $S_{\gamma_i}$ in $V_n$ such that $\sigma_{\gamma_i}$ induces $\tau_{\gamma_i}$ on $\Sigma_{0,n+1}$. 
	
	\noindent We start with the canonical embedding of $\Sigma_{0,n+1} \times [0,1]$ in $V_n \times [0,1]$ given by $(x,t) \mapsto (\iota(x), t)$. We divide $V_n \times [0,1]$ in $m$ equal parts, $V_n \times [\frac{k}{m}, \frac{k+1}{m}]$ ($0 \leq k \leq m-1$), and apply $\sigma^{\alpha_k}_{S_{i_k}}$ on the fiber $V_n \times \{ \frac{2k+1}{2m} \}$ (for each $k$). This induces an open book embedding of $\ob(\Sigma_{0,n+1}, \phi_M)$ in $\ob(V_n, \prod_{\gamma_1,\dots,\gamma_m \in \mathcal{C}} \sigma^{\alpha_1}_{\gamma_1} \circ \cdots \circ \sigma^{\alpha_m}_{\gamma_m})$. Recall that every $\sigma_{\gamma_i}$ is a composition of twists along the boundary spheres $S_1, S_2, \dots, S_n$ and $\tw(V_n) = (\Z_2)^n$. Therefore, $\prod_{\gamma_1,\dots,\gamma_m \in \mathcal{C}} \sigma^{\alpha_1}_{\gamma_1} \circ \cdots \circ \sigma^{\alpha_m}_{\gamma_m} = \prod_{q=1}^{n} \sigma^{\beta_q}_{S_q}$, for some integers $\beta_1,\dots,\beta_n$. Note that $\ob(V_n, \prod_{q=1}^{n} \sigma^{\beta_q}_{S_q}) = \ob(S_1 \times [0,1], \sigma^{\beta_1}_{S_1}) \# \ob(S_2 \times [0,1], \sigma^{\beta_2}_{S_2}) \# \cdots \# \ob(S_n \times [0,1], \sigma^{\beta_n}_{S_n})$. Thus, $\ob(V_n, \prod_{\gamma_1,\dots,\gamma_m \in \mathcal{C}} \sigma^{\alpha_1}_{\gamma_1} \circ \cdots \circ \sigma^{\alpha_m}_{\gamma_m})$ is diffeomorphic to $W_{i,j} = (\#^i\s^2 \times \s^2) \# (\#^j \s^2 \widetilde{\times}\s^2)$, where $i$ and $j$ are the numbers of even and odd $\beta_q$, respectively, and $i + j = n$. 
		
	\end{proof}

	We now discuss some explicit constructions of codimension $1$ spun embeddings of $3$ and $4$ manifolds.  
	
	\subsection{Lens spaces} Let $L(p,q)$ denote the lens space obtained by a $-\frac{p}{q}$ surgery on an unknot in $\s^3$. Here, $p$ and $q$ are coprime positive integers. Consider the following continued fraction expansion.
	
	\[
	-\frac{p}{q} = a_1 - \cfrac{1}{a_2 - \cfrac{1}{a_3 - \cdots - \cfrac{1}{a_k}}} = [a_1,a_2,\dots,a_k], \ \ a_i \leq -2, \ \  i = 1, 2, \dots, k.
	\]
	
	The lens space $L(p,q)$ is obtained from a linear plumbing tree with $k$ vertices with labels $a_1,\dots, a_k$. See the upper diagram in Figure \ref{figlensplanar}

	\begin{figure}[htbp] 
		
		\centering
		\def\svgwidth{11cm}
		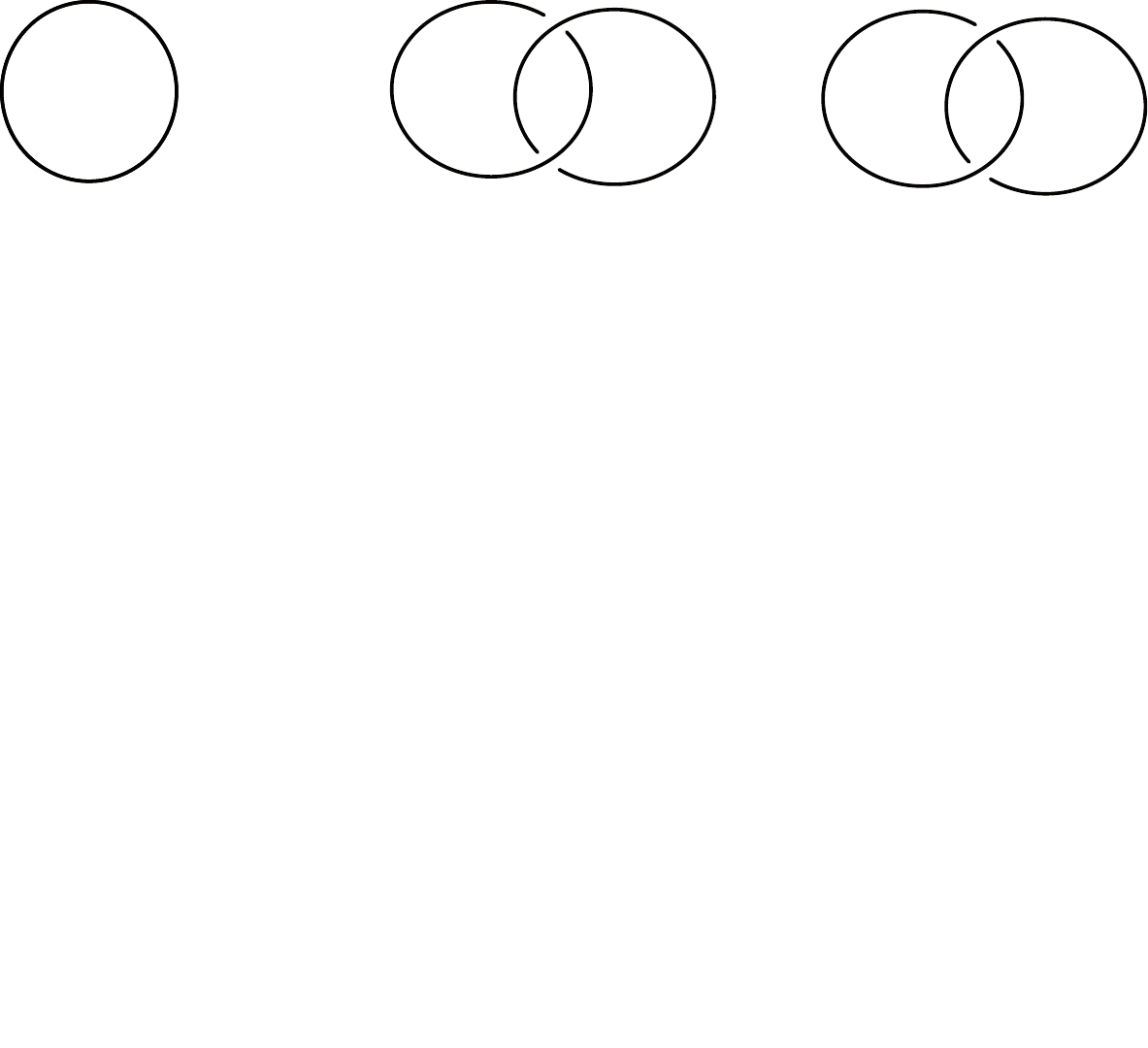
		\caption{}
		\label{figlensplanar}
		
	\end{figure}

	\begin{theorem} \label{lensembed}
		
	Let $-\frac{p}{q} = [a_1,a_2,\dots,a_k]$. Then $L(p,q)$ spun embeds in $\#^k \s^2 {\times} \s^2$, if $a_i( 1\leq i \leq k)$ is even and spun embeds in $\#^k \s^2 \widetilde{\times} \s^2$ otherwise.
		
	\end{theorem}

    \begin{proof}
    	
     We will use the construction of planar open books from surgery diagrams of lens spaces as given by Sch{\"o}nenberger (section $3.1$ in \cite{Sch}). One starts with the linear plumbing diagram corresponding to the continued fraction $[a_1,a_2,\dots,a_k]$, and then perform handle slides. First, sliding the unknot with label $a_2$ along the unknot with label $a_1$ changes its surgery coefficient from $a_2$ to $b_2 = a_1 + a_2 + 2$ and the linking number between the unknots with label $b_1 = a_1$ and $b_2$ is $a_1 + 1$. Then, we slide the unknot with label $a_3$ over the unknot with label $a'_2$. This changes $a_3$ to $b_3 = a_1 + a_2 + a_3 + 4$. The linking number between unknots with labels $b_1$ and $b_3$ is $a_1 + 1$, and the linking number between unknots with labels $b_2$ and $b_3$ is $a_2 +1$. Continuing like this, we get a surgery diagram consisting of a closed $k$-braid with framings (as labels) $b_1, b_2,\dots,b_k$ on its strands, where $b_i = \sum_{1}^{i}a_j + 2(i-1)$. Moreover, there is a $a_i + 2$-twist around the strands $b_i, b_{i+1},\dots, b_k$ for all $i \in \{2,\dots,k-1\}$ and a $a_1 + 1$-twist around the strands $b_1, b_2, \dots, b_k$. See the lower picture in Figure \ref{figlensplanar}. 
     \noindent We recall the construction of planar open book decomposition discussed in section \ref{planarob}. A similar application of the \emph{Rolfsen twist} and the \emph{blow up} operation on the surgery diagram gives us a planar open book decomposition of $L(p,q)$ with page $\Sigma_{0,k+1}$ and monodromy $\psi$. Here, $\psi$ is a composition of Dehn twists along the boundary parallel curves $\gamma_1, \gamma_2,\dots,\gamma_k$ and the simple closed curves $\gamma_{ij}$ homologous to $\sum_{j=i}^{k}[\gamma_j]$ for $i \in \{1,2,\dots,k-1\}$. In fact, the strand with surgery coefficient $b_i$ contributes $b_i$ copies of positive Dehn twist along $\gamma_i$ in the presentation of $\psi$. The $a_i +2$-twist contributes $a_i + 2$ copies of positive Dehn twists along the simple closed curve $\gamma_{i,i+1,\dots,k}$ bounding the labelled strands $b_i, b_{i+1},\dots,b_k$ for $i \in \{2,\dots,k\}$, and there are $a_1 + 1$ copies of positive Dehn twists along the simple closed curve $\gamma_{1,2,\dots,k}$ bounding the strands $b_1,b_2,\dots,b_k$. See Figure \ref{lplr}.
     
      \begin{figure}[htbp]

     	\centering
     	\def\svgwidth{9cm}
     	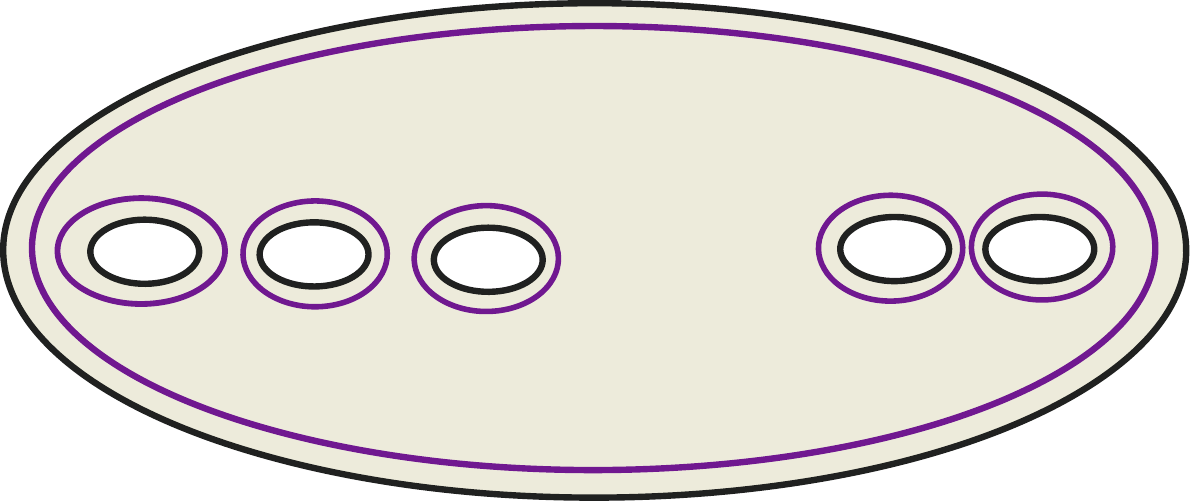
     	\caption{}
     	\label{lplr}
     	
     \end{figure} 
      
     \noindent Similar to the proof of Theorem \ref{thm0}, we properly embed $\Sigma_{0,k+1}$ in $M_k = \#^k_b (\s^2 \times [0,1])$ such that twist along the $i$th sphere, $\sigma_i$, induces Dehn twist along $\gamma_i$. Let $\sigma_{i,i+1,\dots,k}$ denote the twist along the sphere obtained by tubing the spheres with labels $i,i+1,\dots,k$ ($i \geq 1$). Then $\sigma_{i,i+1,\dots,k}$ induces the Dehn twist along $\gamma_{i,i+1,\dots,k}$. Since $Twist(M_k) = (\Z_2)^k$ and $\sigma_{i,i+1,\dots,k} = \sigma_i \circ \sigma_{i+1}\circ \cdots \circ \sigma_k$, $\ob(\Sigma_{0,k+1}, \psi)$ spun embeds in $\ob(M_k, \Psi)$, where
      
      \begin{align*}
      \Psi &= \sigma_1^{ a_1+2} \circ \sigma_2^{ a_1 + a_2+4} \circ \cdots \circ \sigma_{k-1}^{\sum_{i=1}^{k-1}a_i+2(k-1)} \circ \sigma_k^{\sum_{j=1}^ka_j+2k}  \\
      	&= \sigma_1^{a_1} \circ \sigma_2 ^{ a_1 + a_2} \circ \cdots \circ \sigma_{k-1}^{\sum_{i=1}^{k-1}a_i} \circ \sigma_k^{\sum_{j=1}^ka_j}\\
       \end{align*}
      
      \vspace{0.25cm}
      
      \noindent For $a_i (1\leq i \leq k)$ even , $\ob(M_k, \Psi) = \#^k \s^2 {\times} \s^2$. If for some $i$, $a_i$ odd, $\ob(M_k,\Psi) =  \#^k \s^2 \widetilde{\times} \s^2$, as $\s^2 \widetilde{\times} \s^2 \# \s^2 \times \s^2 = \s^2 \widetilde{\times} \s^2 \# \s^2 \widetilde{\times} \s^2$.
     
    \end{proof}

\begin{corollary} \label{cor1}
	
	Let $L(p,1)$ denote the lens space obtained by a $p$ surgery on $\s^3$.
	
	\begin{enumerate}
		
		\item For $p$ even, $L(p,1)$ open book embeds in $\s^2 \times \s^2$. 
		
		\item For $p$ odd, $L(p,1)$ open book embeds in $\s^2 \widetilde{\times} \s^2$.
		
	\end{enumerate}

\end{corollary}
  \begin{exmp}
	The lens space $L(pq-1,q)$, has a surgery diagram as given in Figure \ref{figS03}. It admits a planar open book decomposition $\ob(\Sigma_{0,3},\tau_a^p\circ \tau_b^{p+q-2}\circ \tau_c^{p-1})$. It admits a spun embedding in $\#^2(\s^2 \widetilde{\times}\s^2)$.
	
\end{exmp}

     \begin{figure}[htbp]

      	\centering
      	\def\svgwidth{9cm}
      	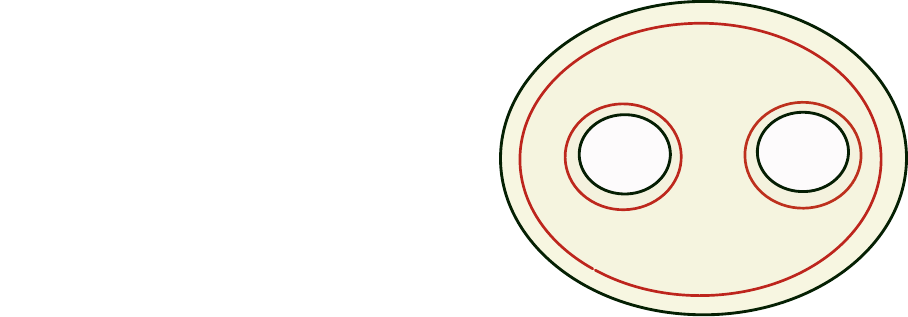
      	\caption{}
      	\label{figS03}
      	
      \end{figure}

    \subsection{Small Seifert fibered spaces}
    
    The methods used in Theorem \ref{lensembed} can also be used to produce spun embedding of small Seifert fibered spaces. A small Seifert manifold $M(e_0;r_1,r_2,r_3)$ ($e_0 \in \Z$ and $r_1,r_2,r_3 \in (0,1) \cap \Q$) has the following surgery diagram as shown in Figure ***.
    
    \begin{exmp}
    The following example is taken from \cite{GLS}. Let $M_p = M(-1;\frac{1}{2}, \frac{1}{2},\frac{1}{p})$ ($p \geq 2)$. As one can see from the surgery description of $M_p$ (see Figure \ref{figMp}) that it has an open book given by $\ob(\Sigma_{0,6},\tau_{a_1} \circ\tau_{a_2}\circ \tau_{a_3}^2\circ \tau_{a_4}^2 \circ \tau_{a_5}^{p}\circ \tau_b)$. We properly embed $\Sigma_{0,6}$ in $M_5 = \#^5(\s^2 \times [0,1])$. As before, we index the $2$-spheres by $1$ to $5$ and let $\sigma_i$ denote the twist along the $i$th sphere and let $\sigma_{S_{i_1\dots i_k}}$ denote the twist along an embedded sphere homologous to the sum of the spheres indexed by $i_1,i_2,\dots,i_k$. Thus, $M_p$ spun embeds in $\ob(M_5, \sigma_1 \circ \sigma_2 \circ \sigma_3^2 \circ \sigma_4^2 \circ \sigma_5^{p}\circ \sigma_{12345}) = \ob(M_5, \sigma_1^2 \circ \sigma_2^2 \circ \sigma_3^3 \circ \sigma_4^3 \circ \sigma_5^{p+1}) = \#^5 \s^2 \widetilde{\times} \s^2$.  
    
    \begin{figure}[htbp] 
    	
    	\centering
    	\def\svgwidth{14cm}
    	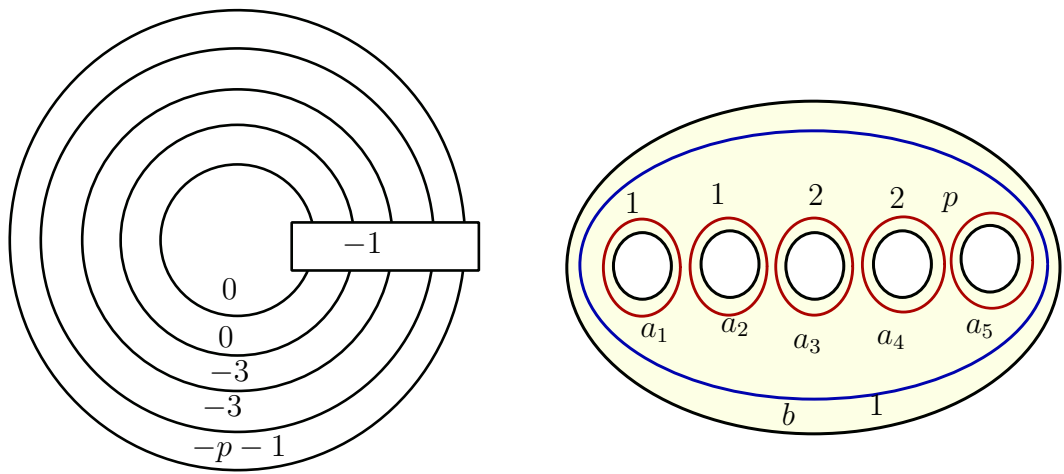
    	\caption{Planar open book for $M_p$ from braided surgery diagram.}
    	\label{figMp}
    	
    \end{figure}
    
    \end{exmp}

    \begin{exmp}
    	
    	Figure \ref{figtripodplumb} describes the plumbing diagram for the Seifert fibered space $M(-3;-\frac{3}{2},-\frac{5}{3},-\frac{5}{3})$. Using surgery moves we obtain a planar open book for this manifold $\ob(\Sigma_{0,6}, \phi)$, where $\phi = \tau_{a_1} \circ \tau^{-2}_{a_2} \circ \tau^{-1}_{a_3}\circ \tau^3_{a_4} \circ \tau_{a_5}\circ \tau^2_{a_6} \circ \tau^3_{b_1} \circ \tau^2_{b_4}\circ \tau^{-2}_{b_3} \circ \tau_{b_2}$ (see Figure \ref{figex}). We can spun embed this in $(\#^2\s^2 \times \s^2) \# (\#^4 \s^2 \widetilde{\times} \s^2)$.
    	
    	 \begin{figure}[htbp] 
    	 	\centering
    		\def\svgwidth{8cm}
    		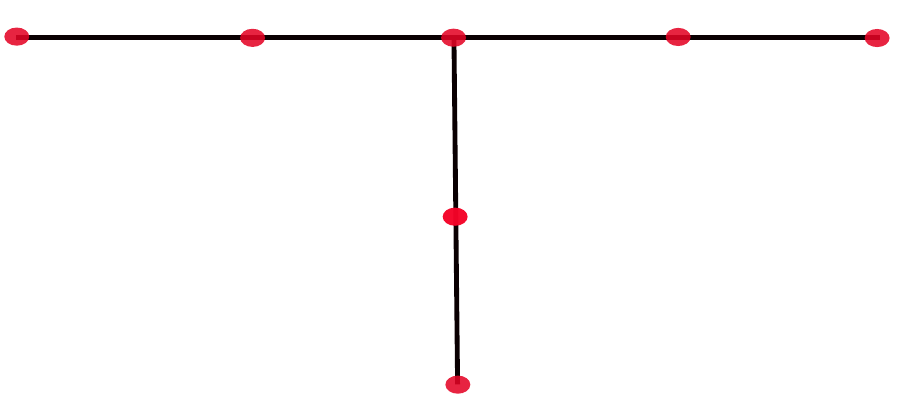
    		\caption{The plumbing diagram for $M(-3;-\frac{3}{2},-\frac{5}{3},-\frac{5}{3})$.}
    		\label{figtripodplumb}
    		
    	\end{figure}

    	\begin{figure}[htbp] 
    		
    		\centering
    		\def\svgwidth{15cm}
    		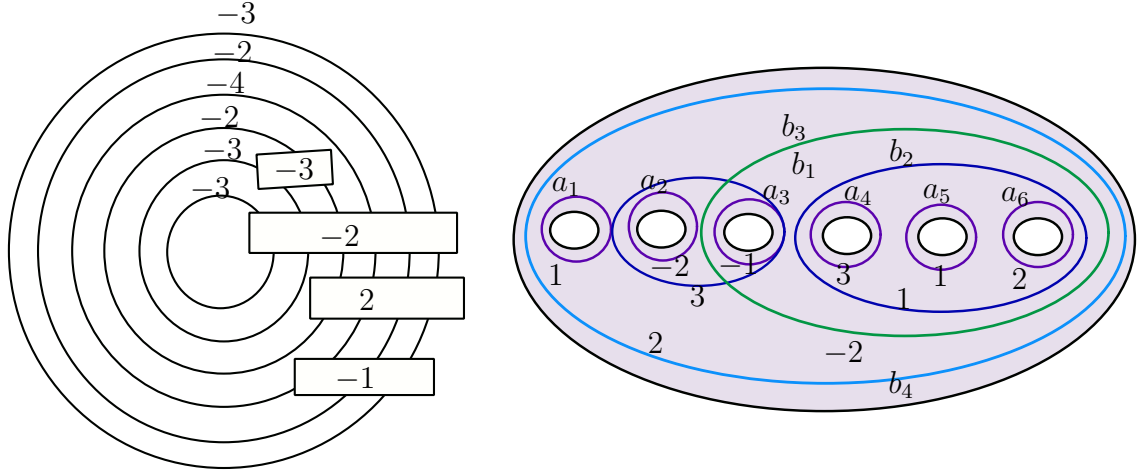
    		\caption{Planar open book for $M(-3;-\frac{3}{2},-\frac{5}{3},-\frac{5}{3})$ from braided surgery diagram.}
    		\label{figex}
    		
    	\end{figure}

    \end{exmp}

    \subsection{Poincaré homology sphere} The famous Poincaré homology sphere $\Sigma(2,3,5)$ has a planar open book decomposition $\ob(\Sigma_{0,3}, \tau_{\alpha}^{-1}\circ\tau_{\beta}^{-1}\circ\tau_{\alpha}\circ\tau_{\beta}\circ\tau_{\delta_1}^{-1} \circ \tau_{\delta_2} \circ \tau_{\delta_3}^{-1})$(see Figure \ref{phsphere}). Thus, $\Sigma(2,3,5)$ spun embeds in $\ob(M_3,\sigma_1 \circ \sigma_2 \circ \sigma_3) = \#^3 \s^2 \widetilde{\times}\s^2 = \s^2 \times \s^2 \# \s^2 \times \s^2 \# \s^2 \widetilde{\times} \s^2$. 
    
    \begin{figure}[htbp] 
    	
    	\centering
    	\def\svgwidth{7cm}
    	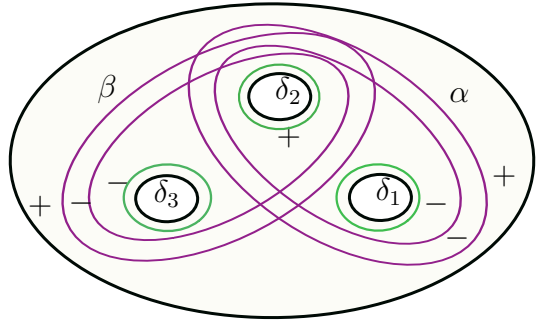
    	\caption{A planar open book for $\Sigma(2,3,5)$}
    	\label{phsphere}
    	
    \end{figure}
    
    \noindent It was shown by Aceto, Golla and Larson \cite{AGL} that  $\Sigma(2,3,5)$ embeds in $\#^8(\s^2 \times \s^2)$ and $\Sigma(2,3,5)$ cannot be embedded in $\#^k(\s^2 \times \s^2)$ for $k \leq 7$. We show below that $\Sigma(2,3,5)$ actually spun embeds in $\#^8(\s^2 \times \s^2)$.

       \begin{figure}[!htb]\label{}
	       \centering
	       \def\svgwidth{10cm}
           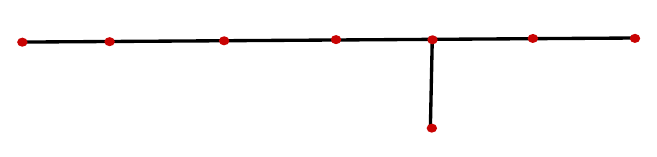
	       \vspace{-0.1cm}
	       \caption{ }
	        \label{ph0}
         \end{figure} 
      Figure \ref{ph0} gives the plumbing diagram for $\Sigma(2,3,5)$. Using the recipe given in Sch{\"o}nenberger (See proof of theorem $3.2.1$ in \cite{Sch}) we obtain surgery diagram of $\Sigma(2,3,5)$ shown on the left of Figure \ref{ph1}.
      
      \begin{figure}[!htb]\label{}
      	\centering
      	\def\svgwidth{12cm}
      	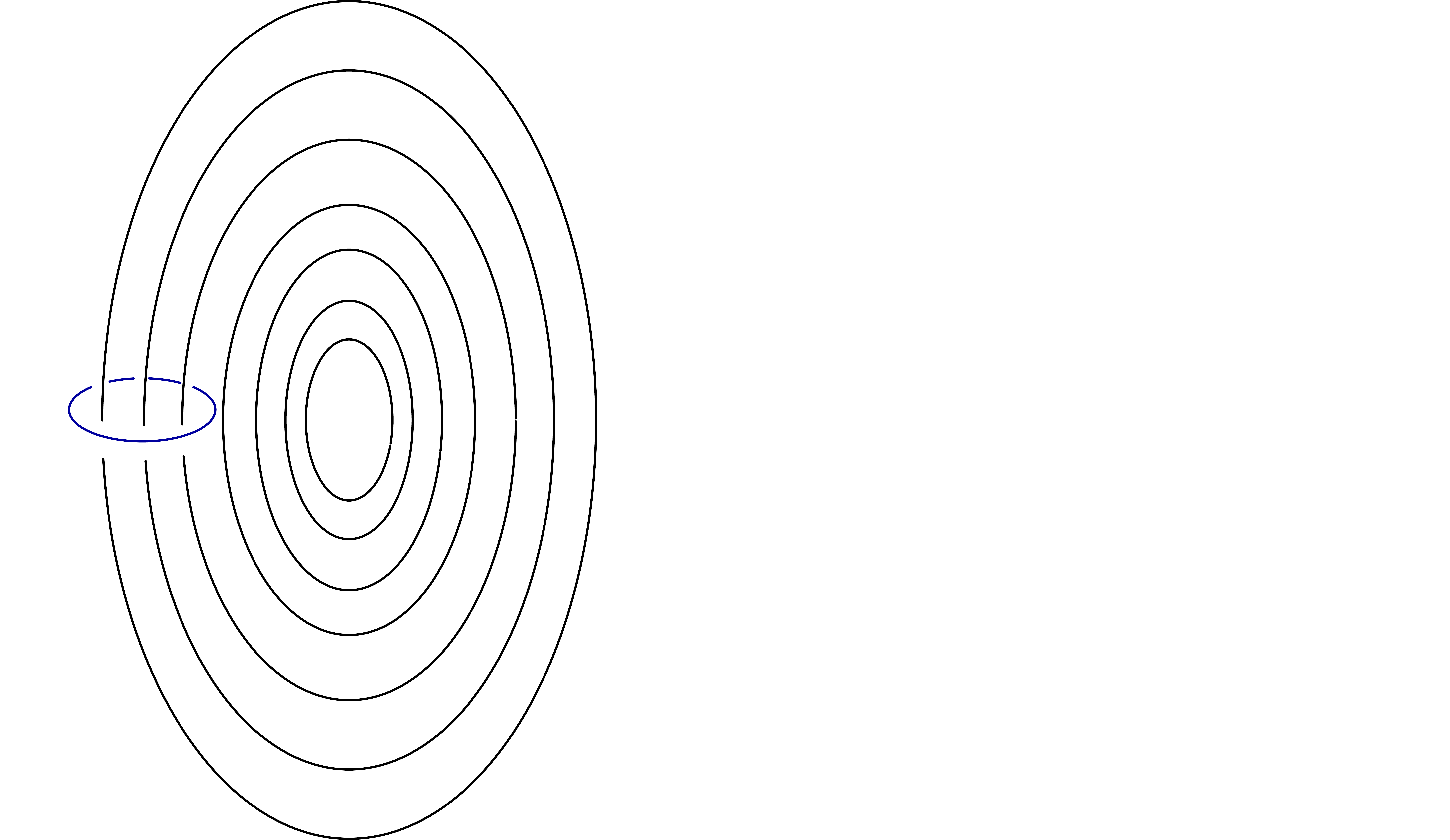
      	\vspace{-0.1cm}
      	\caption{ }
      	\label{ph1}
      \end{figure} 
      
      We perform a blow up operation to remove each linking. See diagram on right of the Figure \ref{ph1}. Thus surgery coefficients of all unknots changes from $-2$ to $-1$. Now slide the  blue unknot labelled $-2$ along the outermost unknot labelled  with $-1$, as on the right of the Figure \ref{ph1}.Thus blue unknot has new surgery coefficient $-1.$ After doing isotopy we get a link  shown on left of the Figure \ref{ph*}. Using surgery moves we obtain a planar open book for $\Sigma(2,3,5)$ $\ob(\Sigma_{0,8}, \phi)$,  where $\phi= \tau_{a_1} \circ \tau_{a_2} \circ \tau_{a_3} \circ \tau_{a_4} \circ \tau_{a_5}^2 \circ \tau_{a_6}^3 \circ \tau_{a_7} \circ \tau_{a_8} \circ \tau_{b_1} \circ \tau_{b_2}^{-1} \circ \tau_{b_3}^{-1}$ (Left on the Figure\ref{ph*}). Thus $\Sigma(2,3,5)$ spun embeds in $\ob(M_8,\sigma_{a_1} \circ \sigma_{a_2} \circ \sigma_{a_3} \circ \sigma_{a_4} \circ \sigma_{a_5}^2 \circ \sigma_{a_6}^3 \circ \sigma_{a_7} \circ \sigma_{a_8} \circ \sigma_{b_1} \circ \sigma_{b_2} \circ \sigma_{b_3} )=\ob(M_8, \sigma_{a_1}^2 \circ \sigma_{a_2}^2 \circ \sigma_{a_3}^2 \circ \sigma_{a_4}^2 \circ \sigma_{a_5}^4 \circ \sigma_{a_6}^6 \circ \sigma_{a_7} ^4\circ \sigma_{a_8}^2) =\#^8 \s^2 \times \s^2.$

  \begin{figure}[!htb]\label{}
  	\centering
  	\def\svgwidth{13cm}
  	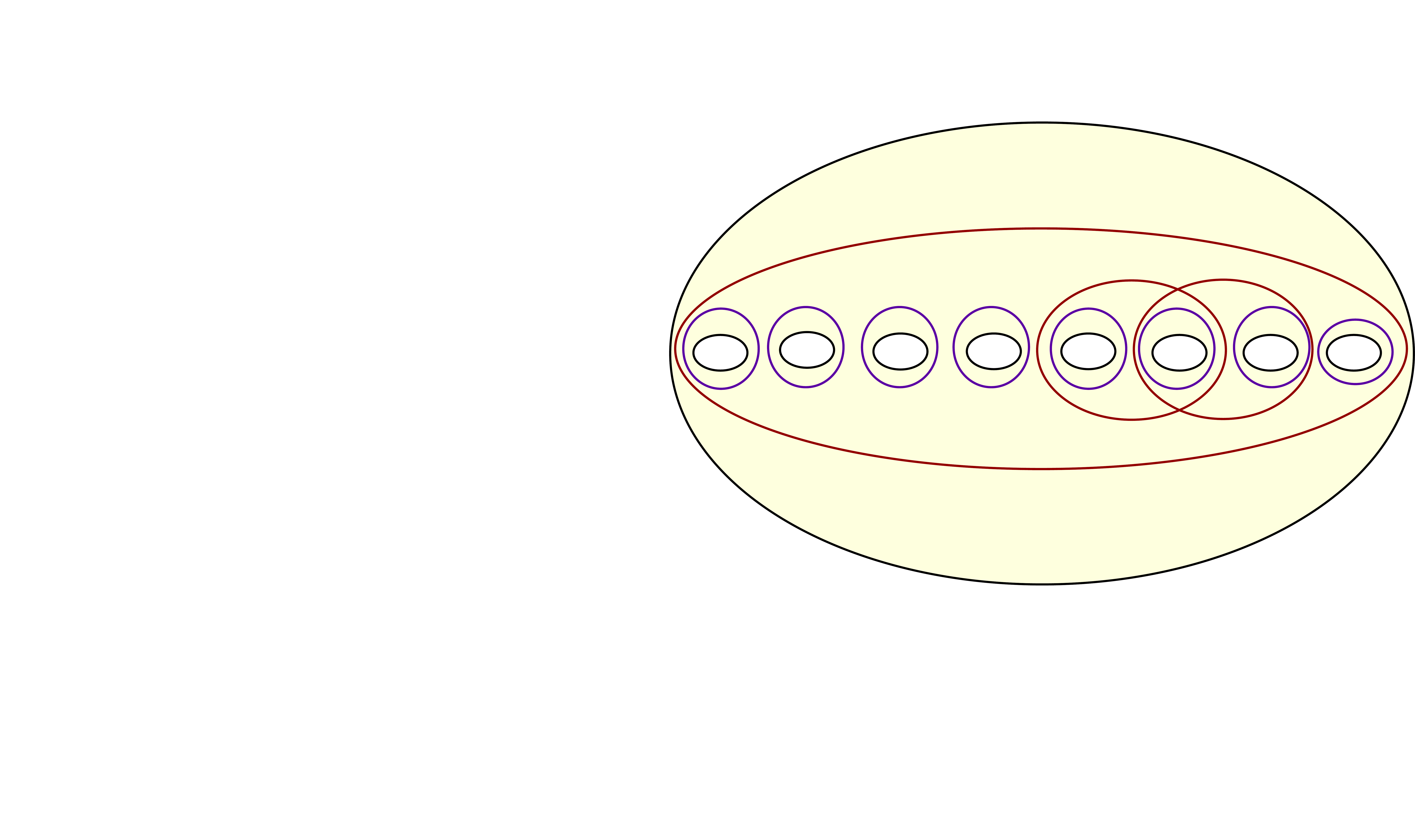
  	\vspace{-0.1cm}
  	\caption{ }
  	\label{ph*}
  \end{figure} 

    \subsection{Codimension--1 spun embedding in $\s^2 \times \s^2$}
    It is well-known that every closed oriented $3$-manifold embeds in $\#^n(\s^2 \times \s^2)$ for some $n$. One can use spun embedding to produce examples of $3$-manifolds that embeds in connected sums of $\s^2 \times \s^2$.

     \begin{exmp}\label{exmplek}
    	
    	Consider an open book with page $\Sigma_{0,3}$ and monodromy $h=\tau_a^{i_a}\tau_b^{i_b}\tau_c^{i_c}\tau_d^{i_d}\tau_e^{i_e}\tau_f^{i_f}\tau_g^{i_g}$ (see Figure \ref{figlek}) such that $i_a+i_d+i_e+i_g = i_b+i_d+i_e+i_f = i_c+i_d+i_f+i_g = 0 \pmod 2$. It follows from the proof of Theorem \ref{thm0} that $\ob(\Sigma_{0,3},h)$ spun embeds in $\ob(M_3, \sigma_a^{i_a+i_d+i_e+i_g}\circ\sigma_b^{i_b+i_d+i_e+i_f}\circ\sigma_c^{i_c+i_d+i_f+i_g}) = \#^3 (\s^2 \times \s^2)$.
    	
    	\begin{figure}[htbp] 
    		
    		\centering
    		\def\svgwidth{7cm}
    		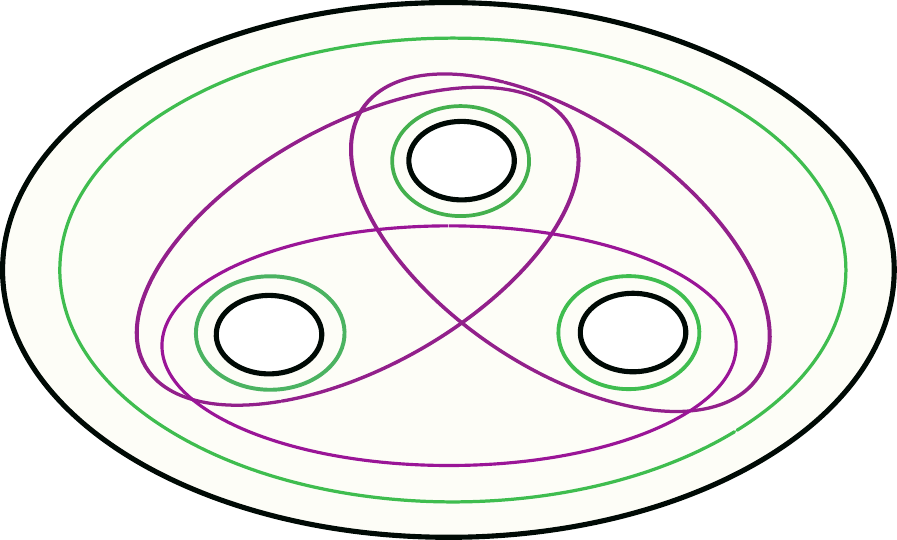
    		\caption{}
    		\label{figlek}
    		
    	\end{figure}

    \end{exmp}
    
    In general, let $M^3 = \ob(\Sigma_{0,n+1}, \phi_M)$ be a planar open book. Let $\phi_M = \prod_{\gamma_1,\dots,\gamma_m \in \mathcal{C}} \tau^{\alpha_1}_{\gamma_1} \circ \cdots \circ \tau^{\alpha_m}_{\gamma_m}$, for some integers $m, \alpha_1,\dots \alpha_m$. Let $\beta_1, \beta_2,\dots, \beta_n$ be the boundary parallel simple closed curves (in the interior). If $[\gamma_i] = \sum_{j=1}^n c_{ij} [\beta_j]$ and $n_j = \sum_{i=1}^{m} c_{ij} \alpha_i = 0 \pmod 2$ for all $j \in \{1,2,\dots,m\}$ then, $\ob(\Sigma_{0,n+1}, \phi_M)$ spun embeds in $\#^n(\s^2 \times \s^2)$.

    \subsection{Codimension--1 spun embedding in $\s^4$}  Consider a proper embedding of $\Sigma_{0,3}$ in $H_{1,1} = \s^1 \times \D^2 \#_b \s^2 \times [0,1]$ as described in Figure \ref{FIGST}. Then, the push map on $H_{1,1}$ induces the push map on the embedded $\Sigma_{0,3}$. If $a$ and $b$ denote the boundary parallel curves as in Figure \ref{FIGST} then the induced push map is given by the mapping class $\tau_a\circ \tau_b^{-1}$.

	\begin{figure}[htbp] 
		
		\centering
		\def\svgwidth{12cm}
		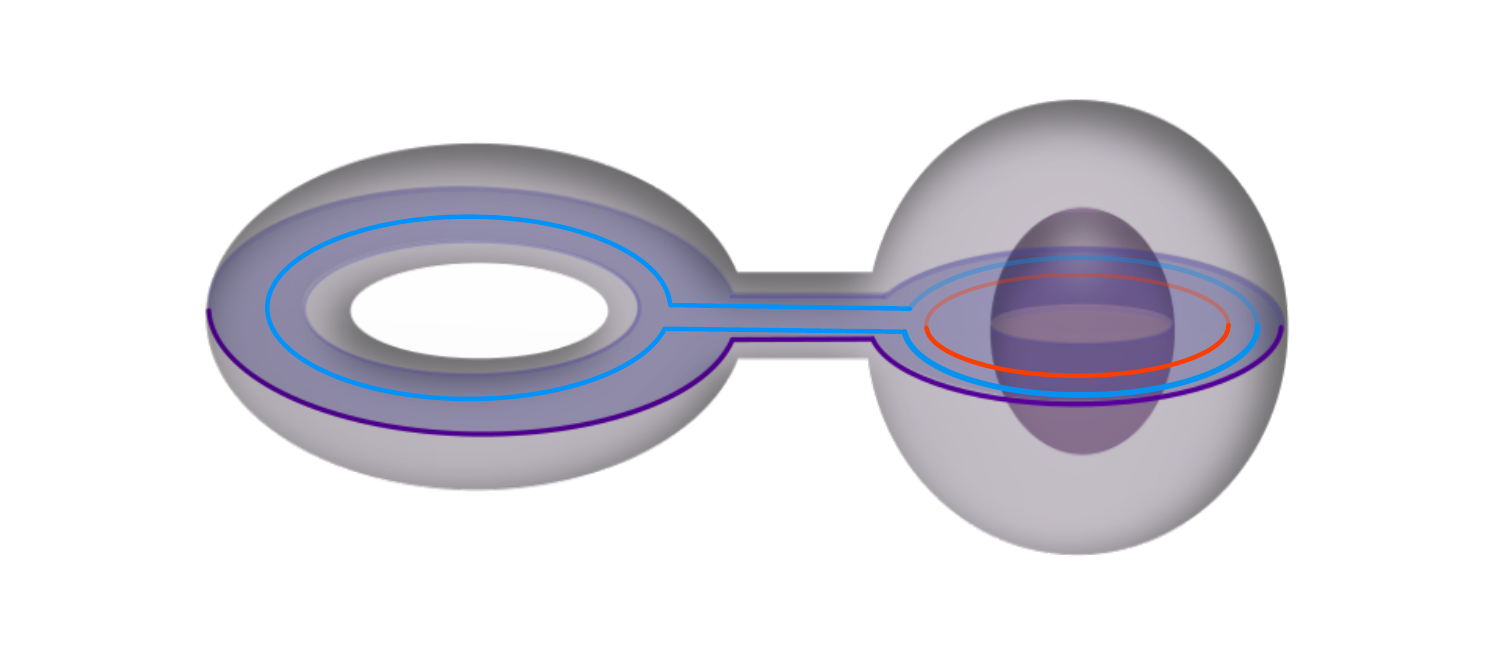
		\caption{}
		\label{FIGST}
		
	\end{figure}

    \begin{exmp}\label{spunembeds4}
    	
    	Figure \ref{figembeds4} gives a family of planar open books of $\s^3$ : $\ob(\Sigma_{0,3}, \tau_c^{-2k-2} \circ \tau_a^{-1}\circ \tau_b)$, $k\geq0$. We embed $\Sigma_{0,3}$ in $H_{1,1}$ as described in Figure \ref{FIGST}. Consider the induced product embedding of $\Sigma_{0,3} \times [0,1] \rightarrow H_{1,1} \times [0,1]$. Let $\sigma_c$ denote the twist along the boundary $2$-sphere of $H_{1,1}$ and $\rho_{b,a}$ denote the push map that pushes the boundary sphere in $H_{1,1}$ along the curve $a$. We then apply $\sigma^{-2k-2}_c$ at level $H_{1,1} \times \{\frac{1}{3}\}$ and $\rho_{b,a}$ at level $H_{1,1} \times \{\frac{2}{3}\}$. This will give a spun embedding of $\ob(\Sigma_{0,3}, \tau_c^{2k+2} \circ \tau_a\circ \tau^{-1}_b)$ in $\ob(H_{1,1}, \sigma_c^{2k+2}\circ \rho_{b,a}) \cong \ob(H_{1,1},\rho_{b,a}) = \s^4$.

    	  \begin{figure}[htbp] 
    		
    		\centering
    		\def\svgwidth{12cm}
    		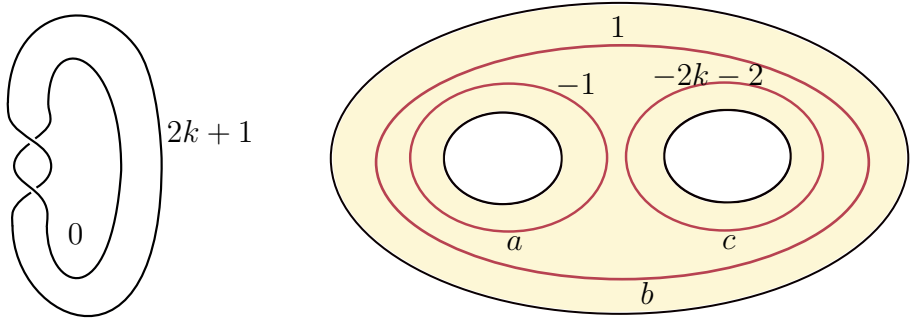
    		\caption{A family of planar open books for $\s^3$.}
    		\label{figembeds4}
    		
    	\end{figure}

    	\noindent In fact for any integer $n$, $\ob(\Sigma_{0,3}, \tau_c^{n+1} \circ \tau_a\circ \tau^{-1}_b)$ embeds in $\ob(H_{1,1}, \sigma_c^{n+1}\circ \rho_{b,a}).$ By \cite{hsueh1}, $\ob(H_{1,1}, \sigma_c^{n+1}\circ \rho_{b,a})=\s^4$.
    \end{exmp}
    
    \noindent A general case can be thought of in the following way. Let $M^3 = \ob(\Sigma_{0,2n+1}, \phi_M)$ be a planar open book obtained from a braided surgery diagram. Let $a_1,b_1,a_2,b_2,\dots,a_n,b_n$ be the inner bondary parallel curves in $\Sigma_{0,2n+1}$. We can think of $\Sigma_{0,2n+1}$ as the boundary connected sum $n$ copies of $\Sigma_{0,3}$, where the $i$th copy has inner boundary parallel curves $a_i,b_i$. We can embed each copy of $\Sigma_{0,3}$ in a $H_{1,1}$ as shown in Figure \ref{FIGST} and taking boundary connected sums will give a proper embedding of $\Sigma_{0,2n+1}$ in $\#_b^{2n} H_{1,1}$. Let $\phi_M = \prod_{\gamma_1,\dots,\gamma_m \in \mathcal{C}} \tau^{\alpha_1}_{\gamma_1} \circ \cdots \circ \tau^{\alpha_m}_{\gamma_m} \circ \rho_{b_1,a_1}\circ \rho_{b_2,a_2}\circ \cdots \circ \rho_{b_n,a_n}$, for some integers $m, \alpha_1,\dots \alpha_m$ such that $[\gamma_i] = \sum_{j=1}^n c_{ij} [a_j]$ and $n_j = \sum_{i=1}^{m} c_{ij} \alpha_i = 1 \pmod 2$ for all $j \in \{1,2,\dots,m\}$ then, $\ob(\Sigma_{0,2n+1}, \phi_M)$ spun embeds in $\s^4$.

     \subsection{Codimension--1 spun embedding in $\s^5$}\label{spunembeds5}
     
      Let $N^4 = \ob(M, \id)$, where $M$ is a connected oriented $3$-manifold with non-empty boundary. Let $D(M)$ denote the double of $M$. Saeki has proved the following.
      
      \begin{theorem}(Theorem $6.1$ in \cite{Sk1})\label{thms5} Let $\widetilde{M}$ be a closed orientable connected $3$-manifold. There exists a simple open book decomposition of $\s^5$ with a spin page bounding $\widetilde{M}$. 
      	
      \end{theorem}

      \noindent Thus, there exists an open book decomposition of $\s^5$ with a spin page $V^4$ and binding $\partial V = D(M)$. We identify a collar neighborhood of $D(M)$ in $V$ with $D(M) \times [0,1]$ such that $D(M) \times \{0\}= \partial V$. We construct a proper embedding of $M$ in $D(M) \times [0,1]$ by pushing $M \subset D(M) \times \{0\}$ into the interior of $D(M) \times [0,1]$, fixing the boundary $\partial M$ pointwise. Since the monodromy on page $V$ is identity near the boundary, this gives a spun embedding of $N^4$ in $\s^5$.
      
      \noindent  Hsueh \cite{hsueh1} has shown every $4$ dimensional open book of the form $\ob(M^3, id)$ admits an open book decomposition with page a $3$-d handlebody. Theorem \ref{thms5} shows that the converse is not true, as any $4$ manifold that does not embed in $\s^5$ cannot admit an open book decomposition with identity monodormy.

	\section{Non-triviality of twist map along a non-separating sphere} \label{twistnontriviality}

	\noindent Let $\phi_1 \in \Diff_\partial(\s^1 \times \D^n)$ and $\phi_2 \in \Diff_\partial(\s^n \times[0,1])$. To prove Theorem \ref{thm2} we need the following lemma.

	\begin{lemma} \label{pageplumb}
	
	$\ob(\s^1 \times \D^n \S \s^n \times [0,1], \phi_1\circ \phi_2) = \ob(\s^1 \times \D^n,\phi_1) \# \ob(\s^n \times [0,1],\phi_2)$.
	
	\end{lemma}
	
	\noindent Lemma \ref{pageplumb} is well-known for $n=2$. We follow a proof of the $2$-dimensional case given by Etnyre (see \cite{OP}).

	\begin{figure}[htbp] 
		
		\centering
		\def\svgwidth{14cm}
		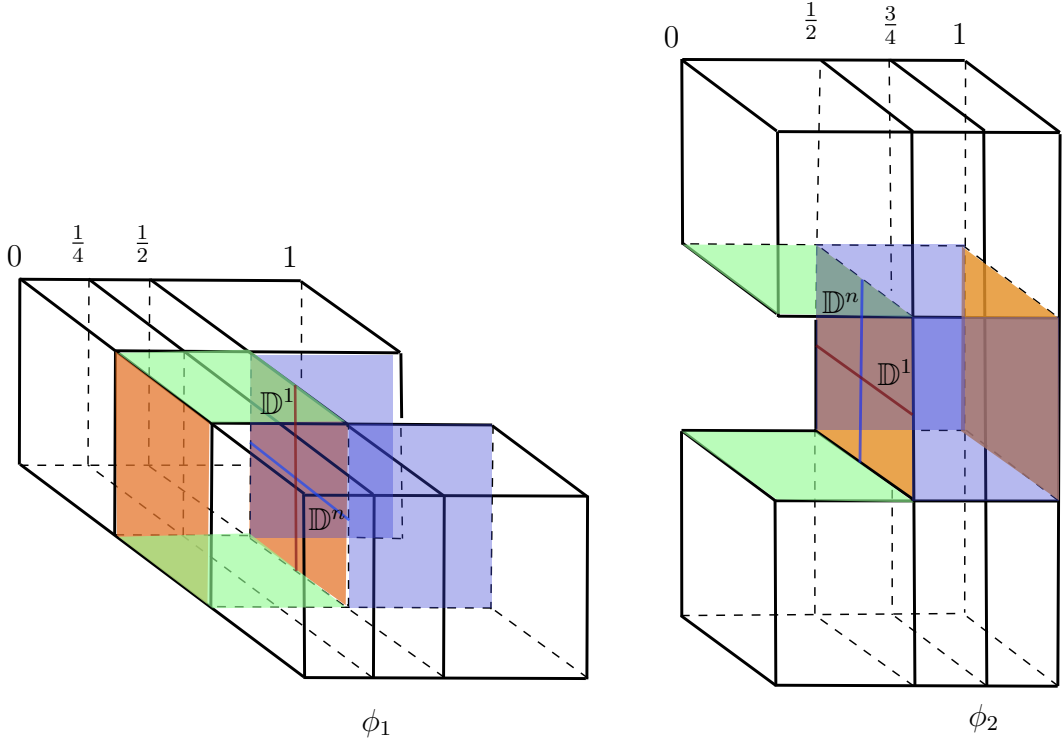
		\caption{The shaded 2d boxes specify the gluing regions after removal of two $(n+1)$-disks from $\ob(\s^1 \times \D^{n-1},\phi_1)$ and $\ob(\s^n \times [0,1],\phi_2)$, respectively.}
		\label{figms}
		
	\end{figure}

\begin{proof}[Proof of Lemma \ref{pageplumb}]
	
	Let us consider neighborhoods of the plumbing regions from both $\s^1 \times \D^n$ and $\s^n \times [0,1]$. During plumbing the two common regions diffeomorphic to $\D^1 \times \D^n$ are identified via the map $(x,\vec{y}) \mapsto (- \vec{y}, x)$. We think of the open book mapping tori as mapping cylinders, where two ends of a cylinder are identified via the identity map. The monodromy in the first mapping cylinder is applied at level $\frac{1}{4}$ of the interval $[0,1]$, whereas the the monodromy in the second mapping cylinder is applied at level $\frac{3}{4}$. To obtained the open book with plumbed pages, we remove $\D^1 \times \D^n \times (\frac{1}{2},1)$ and $\D^n \times \D^1 \times (0,\frac{1}{2})$ from the mapping cylinders and glue them along the shaded regions (see Figure \ref{figms}). Removing $\D^1 \times \D^n \times (\frac{1}{2},1)$ and $\D^n \times \D^1 \times (0,\frac{1}{2})$ removes open $(n+2)$-balls from the total spaces of $\ob(\s^1 \times \D^n,\phi_1)$ and $\ob(\s^n \times [0,1],\phi_2)$. It is enough to show that the shaded gluing regions are actually copies of $\s^{n+1}$. 
	
	\noindent Consider the shaded attaching region in the mapping cylinder shown on the right of Figure \ref{figms}. The green shaded region is given by $\partial \D^n \times \D^1 \times [0,\frac{1}{2}]$. The corner $\partial \D^n \times \partial \D^1 \times [0,1]$ belongs to the binding after plumbing. Thus it contributes the region $\partial \D^n \times \D^1 \times D$, where $D$ is a disk bound by a boundary circle $\{*\} \times \frac{[0,1]}{0 \sim 1}$ in the open book. These two regions gives us $R_1 = \partial \D^n \times \D^1 \times [0, \frac{1}{2}] \cup \partial \D^n \times \partial \D^1 \times D = (\partial \D^n \times \{\D^1 \times ([0,\frac{1}{2}] \cup [\frac{1}{2},1])  \cup \partial \D^1 \times D\}) \setminus (\partial \D^n \times \D^1 \times [\frac{1}{2},1]) = \partial \D^n \times \{\partial(\D^1 \times D) \setminus \D^1 \times [\frac{1}{2},1]\} \cong \s^{n-1} \times \D^2$.
	
	\noindent Now consider the attaching regions shaded violet and orange. These give the region $R_2 = \D^n \times \D^1 \times \{\frac{1}{2},1\} \cup \D^n \times \partial \D^1 \times [\frac{1}{2},1] = \D^n \times \partial(\D^1 \times [\frac{1}{2},1])$. It is clear that $R_1$ and $R_2$ shares a common boundary and $R_1 \cup_\partial R_2 \cong \s^{n-1} \times \D^2 \cup \D^n \times \s^1 = \s^{n+1}$. Since the monodromies $\phi_1$ and $\phi_2$ are applied away from the gluing regions (except near the boundary of pages, where they restrict to the identity map), the lemma follows.		
	
\end{proof}

	\begin{proof}[Proof of Theorem \ref{thm2}]
		
	 We note that $\s^1 \times \s^{n-1} \setminus int(\D^n) = \s^1 \times \D^{n-1} \S \s^{n-1} \times [0,1]$. Then, by Lemma \ref{pageplumb}, $\ob(\s^1 \times \D^{n-1} \S \s^{n-1} \times [0,1],\sigma_{n-1}) = \ob(\s^1 \times \D^{n-1},id) \# \ob(\s^{n-1} \times [0,1], \sigma_{n-1}) = \s^1 \times \s^n \# \s^2 \widetilde{\times} \s^{n-1}$. Let us assume that $\sigma_{n-1} \in Diff^{+}_\partial (\s^1 \times \s^{n-1} \setminus int (\D^n))$ is isotopic to identity. Then, $\ob(\s^1 \times \s^{n-1} \setminus int(\D^n), \sigma_{n-1}) = \s^1 \times \s^n \# \s^2 \times \s^{n-1}$, which embeds in $\R^{n+3}$. On the other hand, since $\s^2 \widetilde{\times} \s^{n-1}$ is not spin, $\s^1 \times \s^n \# \s^2 \widetilde{\times} \s^{n-1}$ does not embed in $\R^{n+3}$. This gives a contradiction. Therefore, $\sigma_{n-1}$ is non-trivial in $Diff^{+}_\partial (\s^1 \times \s^{n-1} \setminus int (\D^n))$. Now, $\ob(\s^1 \times \s^{n-1} \# V, \sigma_{n-1})= \ob((\s^1 \times \s^{n-1} \setminus int (\D^n)) \#_b V, \sigma_{n-1}) = \s^1 \times \s^n \# \s^2 \widetilde{\times} \s^{n-1} \# \partial (V \times \D^2)$. Since $\partial (V \times \D^2)$ embeds in $\R^{n+3}$, a similar argument as above implies that $\sigma_{n-1}$ is non-trivial in $Diff^{+}_\partial (\s^1 \times \s^{n-1} \# V)$.
		
	\end{proof}

\end{document}

%% file: stwist.pdf_tex
\begingroup%
  \makeatletter%
  \providecommand\color[2][]{%
    \errmessage{(Inkscape) Color is used for the text in Inkscape, but the package 'color.sty' is not loaded}%
    \renewcommand\color[2][]{}%
  }%
  \providecommand\transparent[1]{%
    \errmessage{(Inkscape) Transparency is used (non-zero) for the text in Inkscape, but the package 'transparent.sty' is not loaded}%
    \renewcommand\transparent[1]{}%
  }%
  \providecommand\rotatebox[2]{#2}%
  \newcommand*\fsize{\dimexpr\f@size pt\relax}%
  \newcommand*\lineheight[1]{\fontsize{\fsize}{#1\fsize}\selectfont}%
  \ifx\svgwidth\undefined%
    \setlength{\unitlength}{338.70772102bp}%
    \ifx\svgscale\undefined%
      \relax%
    \else%
      \setlength{\unitlength}{\unitlength * \real{\svgscale}}%
    \fi%
  \else%
    \setlength{\unitlength}{\svgwidth}%
  \fi%
  \global\let\svgwidth\undefined%
  \global\let\svgscale\undefined%
  \makeatother%
  \begin{picture}(1,1.30408252)%
    \lineheight{1}%
    \setlength\tabcolsep{0pt}%
    \put(0,0){\includegraphics[width=\unitlength,page=1]{stwist.pdf}}%
    \put(0.49120221,0.35680266){\color[rgb]{0,0,0}\makebox(0,0)[lt]{\lineheight{1.25}\smash{\begin{tabular}[t]{l}$\vdots$\end{tabular}}}}%
    \put(0,0){\includegraphics[width=\unitlength,page=2]{stwist.pdf}}%
  \end{picture}%
\endgroup%

%% file: figaro1.pdf_tex
\begingroup%
  \makeatletter%
  \providecommand\color[2][]{%
    \errmessage{(Inkscape) Color is used for the text in Inkscape, but the package 'color.sty' is not loaded}%
    \renewcommand\color[2][]{}%
  }%
  \providecommand\transparent[1]{%
    \errmessage{(Inkscape) Transparency is used (non-zero) for the text in Inkscape, but the package 'transparent.sty' is not loaded}%
    \renewcommand\transparent[1]{}%
  }%
  \providecommand\rotatebox[2]{#2}%
  \newcommand*\fsize{\dimexpr\f@size pt\relax}%
  \newcommand*\lineheight[1]{\fontsize{\fsize}{#1\fsize}\selectfont}%
  \ifx\svgwidth\undefined%
    \setlength{\unitlength}{588.88797161bp}%
    \ifx\svgscale\undefined%
      \relax%
    \else%
      \setlength{\unitlength}{\unitlength * \real{\svgscale}}%
    \fi%
  \else%
    \setlength{\unitlength}{\svgwidth}%
  \fi%
  \global\let\svgwidth\undefined%
  \global\let\svgscale\undefined%
  \makeatother%
  \begin{picture}(1,0.14994019)%
    \lineheight{1}%
    \setlength\tabcolsep{0pt}%
    \put(0,0){\includegraphics[width=\unitlength,page=1]{figaro1.pdf}}%
    \put(0.14454455,0.02382568){\color[rgb]{0,0,0}\makebox(0,0)[lt]{\lineheight{1.25}\smash{\begin{tabular}[t]{l}$\alpha$\end{tabular}}}}%
    \put(0.48184591,0.08508242){\color[rgb]{0,0,0}\makebox(0,0)[lt]{\lineheight{1.25}\smash{\begin{tabular}[t]{l}$\rho_\alpha$\end{tabular}}}}%
  \end{picture}%
\endgroup%

%% file: figpage.pdf_tex
\begingroup%
  \makeatletter%
  \providecommand\color[2][]{%
    \errmessage{(Inkscape) Color is used for the text in Inkscape, but the package 'color.sty' is not loaded}%
    \renewcommand\color[2][]{}%
  }%
  \providecommand\transparent[1]{%
    \errmessage{(Inkscape) Transparency is used (non-zero) for the text in Inkscape, but the package 'transparent.sty' is not loaded}%
    \renewcommand\transparent[1]{}%
  }%
  \providecommand\rotatebox[2]{#2}%
  \newcommand*\fsize{\dimexpr\f@size pt\relax}%
  \newcommand*\lineheight[1]{\fontsize{\fsize}{#1\fsize}\selectfont}%
  \ifx\svgwidth\undefined%
    \setlength{\unitlength}{493.25961802bp}%
    \ifx\svgscale\undefined%
      \relax%
    \else%
      \setlength{\unitlength}{\unitlength * \real{\svgscale}}%
    \fi%
  \else%
    \setlength{\unitlength}{\svgwidth}%
  \fi%
  \global\let\svgwidth\undefined%
  \global\let\svgscale\undefined%
  \makeatother%
  \begin{picture}(1,0.68217088)%
    \lineheight{1}%
    \setlength\tabcolsep{0pt}%
    \put(0,0){\includegraphics[width=\unitlength,page=1]{figpage.pdf}}%
    \put(0.15271808,0.46453075){\color[rgb]{0,0,0}\makebox(0,0)[lt]{\lineheight{1.25}\smash{\begin{tabular}[t]{l}$a_1$\end{tabular}}}}%
    \put(0.43807274,0.47145321){\color[rgb]{0,0,0}\makebox(0,0)[lt]{\lineheight{1.25}\smash{\begin{tabular}[t]{l}$a_2$\end{tabular}}}}%
    \put(0.82696009,0.45445513){\color[rgb]{0,0,0}\makebox(0,0)[lt]{\lineheight{1.25}\smash{\begin{tabular}[t]{l}$a_g$\end{tabular}}}}%
    \put(0,0){\includegraphics[width=\unitlength,page=2]{figpage.pdf}}%
    \put(0.19074705,0.02706287){\color[rgb]{0,0,0}\makebox(0,0)[lt]{\lineheight{1.25}\smash{\begin{tabular}[t]{l}$\theta_1$\end{tabular}}}}%
    \put(0.45425224,0.02007721){\color[rgb]{0,0,0}\makebox(0,0)[lt]{\lineheight{1.25}\smash{\begin{tabular}[t]{l}$\theta_2$\end{tabular}}}}%
    \put(0.7416365,0.02083539){\color[rgb]{0,0,0}\makebox(0,0)[lt]{\lineheight{1.25}\smash{\begin{tabular}[t]{l}$\theta_k$\end{tabular}}}}%
    \put(0,0){\includegraphics[width=\unitlength,page=3]{figpage.pdf}}%
  \end{picture}%
\endgroup%

%% file: braid.pdf_tex
\begingroup%
  \makeatletter%
  \providecommand\color[2][]{%
    \errmessage{(Inkscape) Color is used for the text in Inkscape, but the package 'color.sty' is not loaded}%
    \renewcommand\color[2][]{}%
  }%
  \providecommand\transparent[1]{%
    \errmessage{(Inkscape) Transparency is used (non-zero) for the text in Inkscape, but the package 'transparent.sty' is not loaded}%
    \renewcommand\transparent[1]{}%
  }%
  \providecommand\rotatebox[2]{#2}%
  \newcommand*\fsize{\dimexpr\f@size pt\relax}%
  \newcommand*\lineheight[1]{\fontsize{\fsize}{#1\fsize}\selectfont}%
  \ifx\svgwidth\undefined%
    \setlength{\unitlength}{518.05244348bp}%
    \ifx\svgscale\undefined%
      \relax%
    \else%
      \setlength{\unitlength}{\unitlength * \real{\svgscale}}%
    \fi%
  \else%
    \setlength{\unitlength}{\svgwidth}%
  \fi%
  \global\let\svgwidth\undefined%
  \global\let\svgscale\undefined%
  \makeatother%
  \begin{picture}(1,0.80337905)%
    \lineheight{1}%
    \setlength\tabcolsep{0pt}%
    \put(0,0){\includegraphics[width=\unitlength,page=1]{braid.pdf}}%
    \put(0.57592052,0.38521011){\color[rgb]{0,0,0}\makebox(0,0)[lt]{\lineheight{1.25}\smash{\begin{tabular}[t]{l}$\cdots$\end{tabular}}}}%
    \put(0.86742548,0.38202308){\color[rgb]{0,0,0}\makebox(0,0)[lt]{\lineheight{1.25}\smash{\begin{tabular}[t]{l}$\cdots$\end{tabular}}}}%
    \put(0,0){\includegraphics[width=\unitlength,page=2]{braid.pdf}}%
    \put(0.78051184,0.45845014){\color[rgb]{0,0,0}\makebox(0,0)[lt]{\lineheight{1.25}\smash{\begin{tabular}[t]{l} pure $n$-braid\end{tabular}}}}%
    \put(0,0){\includegraphics[width=\unitlength,page=3]{braid.pdf}}%
    \put(0.75731687,0.10748872){\color[rgb]{0,0,0}\makebox(0,0)[lt]{\lineheight{1.25}\smash{\begin{tabular}[t]{l}$\ddots$\end{tabular}}}}%
    \put(0.4164671,0.35093069){\color[rgb]{0,0,0}\makebox(0,0)[lt]{\lineheight{1.25}\smash{\begin{tabular}[t]{l}$U$\end{tabular}}}}%
    \put(0.69741846,0.52705223){\color[rgb]{0,0,0}\makebox(0,0)[lt]{\lineheight{1.25}\smash{\begin{tabular}[t]{l}$\pm1$\end{tabular}}}}%
    \put(0.70225237,0.58424974){\color[rgb]{0,0,0}\makebox(0,0)[lt]{\lineheight{1.25}\smash{\begin{tabular}[t]{l}$\pm1$\end{tabular}}}}%
    \put(0.70402078,0.64925549){\color[rgb]{0,0,0}\makebox(0,0)[lt]{\lineheight{1.25}\smash{\begin{tabular}[t]{l}$\pm1$\end{tabular}}}}%
    \put(0.69340989,0.71890439){\color[rgb]{0,0,0}\makebox(0,0)[lt]{\lineheight{1.25}\smash{\begin{tabular}[t]{l}$\pm1$\end{tabular}}}}%
    \put(0.68987281,0.78236228){\color[rgb]{0,0,0}\makebox(0,0)[lt]{\lineheight{1.25}\smash{\begin{tabular}[t]{l}$\pm1$\end{tabular}}}}%
    \put(0.71167459,0.68177676){\color[rgb]{0,0,0}\makebox(0,0)[lt]{\lineheight{1.25}\smash{\begin{tabular}[t]{l}$\vdots$\end{tabular}}}}%
    \put(0,0){\includegraphics[width=\unitlength,page=4]{braid.pdf}}%
    \put(0.08092332,0.2075688){\color[rgb]{0,0,0}\makebox(0,0)[lt]{\lineheight{1.25}\smash{\begin{tabular}[t]{l}$i$\end{tabular}}}}%
    \put(0.22593976,0.20501298){\color[rgb]{0,0,0}\makebox(0,0)[lt]{\lineheight{1.25}\smash{\begin{tabular}[t]{l}$j$\end{tabular}}}}%
    \put(0.15336109,0.3749539){\color[rgb]{0,0,0}\makebox(0,0)[lt]{\lineheight{1.25}\smash{\begin{tabular}[t]{l}$\cdots$\end{tabular}}}}%
    \put(0.13855805,0.15645843){\color[rgb]{0,0,0}\makebox(0,0)[lt]{\lineheight{1.25}\smash{\begin{tabular}[t]{l}$A_{ij}$\end{tabular}}}}%
  \end{picture}%
\endgroup%

%% file: blowup.pdf_tex
\begingroup%
  \makeatletter%
  \providecommand\color[2][]{%
    \errmessage{(Inkscape) Color is used for the text in Inkscape, but the package 'color.sty' is not loaded}%
    \renewcommand\color[2][]{}%
  }%
  \providecommand\transparent[1]{%
    \errmessage{(Inkscape) Transparency is used (non-zero) for the text in Inkscape, but the package 'transparent.sty' is not loaded}%
    \renewcommand\transparent[1]{}%
  }%
  \providecommand\rotatebox[2]{#2}%
  \newcommand*\fsize{\dimexpr\f@size pt\relax}%
  \newcommand*\lineheight[1]{\fontsize{\fsize}{#1\fsize}\selectfont}%
  \ifx\svgwidth\undefined%
    \setlength{\unitlength}{562.02977196bp}%
    \ifx\svgscale\undefined%
      \relax%
    \else%
      \setlength{\unitlength}{\unitlength * \real{\svgscale}}%
    \fi%
  \else%
    \setlength{\unitlength}{\svgwidth}%
  \fi%
  \global\let\svgwidth\undefined%
  \global\let\svgscale\undefined%
  \makeatother%
  \begin{picture}(1,0.29168083)%
    \lineheight{1}%
    \setlength\tabcolsep{0pt}%
    \put(0,0){\includegraphics[width=\unitlength,page=1]{blowup.pdf}}%
    \put(-0.00070201,0.23726928){\color[rgb]{0,0,0}\makebox(0,0)[lt]{\lineheight{1.25}\smash{\begin{tabular}[t]{l}$n$\end{tabular}}}}%
    \put(0.10708735,0.26193053){\color[rgb]{0,0,0}\makebox(0,0)[lt]{\lineheight{1.25}\smash{\begin{tabular}[t]{l}$n+1$\end{tabular}}}}%
    \put(0.25630849,0.25191887){\color[rgb]{0,0,0}\makebox(0,0)[lt]{\lineheight{1.25}\smash{\begin{tabular}[t]{l}$n$\end{tabular}}}}%
    \put(0.36104899,0.26206119){\color[rgb]{0,0,0}\makebox(0,0)[lt]{\lineheight{1.25}\smash{\begin{tabular}[t]{l}$n-1$\end{tabular}}}}%
    \put(0.11153493,0.14881181){\color[rgb]{0,0,0}\makebox(0,0)[lt]{\lineheight{1.25}\smash{\begin{tabular}[t]{l}$+1$\end{tabular}}}}%
    \put(0.36300245,0.15416693){\color[rgb]{0,0,0}\makebox(0,0)[lt]{\lineheight{1.25}\smash{\begin{tabular}[t]{l}$-1$\end{tabular}}}}%
    \put(0.16113102,0.00688655){\color[rgb]{0,0,0}\makebox(0,0)[lt]{\lineheight{1.25}\smash{\begin{tabular}[t]{l}$(a)$\end{tabular}}}}%
    \put(0.73760432,0.00281928){\color[rgb]{0,0,0}\makebox(0,0)[lt]{\lineheight{1.25}\smash{\begin{tabular}[t]{l}$(b)$\end{tabular}}}}%
    \put(0.72527781,0.14925779){\color[rgb]{0,0,0}\makebox(0,0)[lt]{\lineheight{1.25}\smash{\begin{tabular}[t]{l}$+1$\end{tabular}}}}%
    \put(0.98150818,0.14936528){\color[rgb]{0,0,0}\makebox(0,0)[lt]{\lineheight{1.25}\smash{\begin{tabular}[t]{l}$-1$\end{tabular}}}}%
  \end{picture}%
\endgroup%

%% file: planar.pdf_tex
\begingroup%
  \makeatletter%
  \providecommand\color[2][]{%
    \errmessage{(Inkscape) Color is used for the text in Inkscape, but the package 'color.sty' is not loaded}%
    \renewcommand\color[2][]{}%
  }%
  \providecommand\transparent[1]{%
    \errmessage{(Inkscape) Transparency is used (non-zero) for the text in Inkscape, but the package 'transparent.sty' is not loaded}%
    \renewcommand\transparent[1]{}%
  }%
  \providecommand\rotatebox[2]{#2}%
  \newcommand*\fsize{\dimexpr\f@size pt\relax}%
  \newcommand*\lineheight[1]{\fontsize{\fsize}{#1\fsize}\selectfont}%
  \ifx\svgwidth\undefined%
    \setlength{\unitlength}{498.43888802bp}%
    \ifx\svgscale\undefined%
      \relax%
    \else%
      \setlength{\unitlength}{\unitlength * \real{\svgscale}}%
    \fi%
  \else%
    \setlength{\unitlength}{\svgwidth}%
  \fi%
  \global\let\svgwidth\undefined%
  \global\let\svgscale\undefined%
  \makeatother%
  \begin{picture}(1,0.3816413)%
    \lineheight{1}%
    \setlength\tabcolsep{0pt}%
    \put(0,0){\includegraphics[width=\unitlength,page=1]{planar.pdf}}%
    \put(0.11511548,0.11070586){\color[rgb]{0,0,0}\makebox(0,0)[lt]{\lineheight{1.25}\smash{\begin{tabular}[t]{l}$\delta_1$\end{tabular}}}}%
    \put(0.24780559,0.10254019){\color[rgb]{0,0,0}\makebox(0,0)[lt]{\lineheight{1.25}\smash{\begin{tabular}[t]{l}$\delta_2$\end{tabular}}}}%
    \put(0.3968268,0.10662307){\color[rgb]{0,0,0}\makebox(0,0)[lt]{\lineheight{1.25}\smash{\begin{tabular}[t]{l}$\delta_3$\end{tabular}}}}%
    \put(0.6764968,0.1209128){\color[rgb]{0,0,0}\makebox(0,0)[lt]{\lineheight{1.25}\smash{\begin{tabular}[t]{l}$\delta_{n-1}$\end{tabular}}}}%
    \put(0.83776636,0.11575583){\color[rgb]{0,0,0}\makebox(0,0)[lt]{\lineheight{1.25}\smash{\begin{tabular}[t]{l}$\delta_n$\end{tabular}}}}%
    \put(0.48664777,0.03721594){\color[rgb]{0,0,0}\makebox(0,0)[lt]{\lineheight{1.25}\smash{\begin{tabular}[t]{l}$\delta_{n+1}$\end{tabular}}}}%
    \put(0.53768241,0.17603011){\color[rgb]{0,0,0}\makebox(0,0)[lt]{\lineheight{1.25}\smash{\begin{tabular}[t]{l}$\cdots$\end{tabular}}}}%
  \end{picture}%
\endgroup%

%% file: FIGS.pdf_tex
\begingroup%
  \makeatletter%
  \providecommand\color[2][]{%
    \errmessage{(Inkscape) Color is used for the text in Inkscape, but the package 'color.sty' is not loaded}%
    \renewcommand\color[2][]{}%
  }%
  \providecommand\transparent[1]{%
    \errmessage{(Inkscape) Transparency is used (non-zero) for the text in Inkscape, but the package 'transparent.sty' is not loaded}%
    \renewcommand\transparent[1]{}%
  }%
  \providecommand\rotatebox[2]{#2}%
  \newcommand*\fsize{\dimexpr\f@size pt\relax}%
  \newcommand*\lineheight[1]{\fontsize{\fsize}{#1\fsize}\selectfont}%
  \ifx\svgwidth\undefined%
    \setlength{\unitlength}{535.22972959bp}%
    \ifx\svgscale\undefined%
      \relax%
    \else%
      \setlength{\unitlength}{\unitlength * \real{\svgscale}}%
    \fi%
  \else%
    \setlength{\unitlength}{\svgwidth}%
  \fi%
  \global\let\svgwidth\undefined%
  \global\let\svgscale\undefined%
  \makeatother%
  \begin{picture}(1,0.47271298)%
    \lineheight{1}%
    \setlength\tabcolsep{0pt}%
    \put(0,0){\includegraphics[width=\unitlength,page=1]{FIGS.pdf}}%
  \end{picture}%
\endgroup%

%% file: lensplanar.pdf_tex
\begingroup%
  \makeatletter%
  \providecommand\color[2][]{%
    \errmessage{(Inkscape) Color is used for the text in Inkscape, but the package 'color.sty' is not loaded}%
    \renewcommand\color[2][]{}%
  }%
  \providecommand\transparent[1]{%
    \errmessage{(Inkscape) Transparency is used (non-zero) for the text in Inkscape, but the package 'transparent.sty' is not loaded}%
    \renewcommand\transparent[1]{}%
  }%
  \providecommand\rotatebox[2]{#2}%
  \newcommand*\fsize{\dimexpr\f@size pt\relax}%
  \newcommand*\lineheight[1]{\fontsize{\fsize}{#1\fsize}\selectfont}%
  \ifx\svgwidth\undefined%
    \setlength{\unitlength}{567.44854662bp}%
    \ifx\svgscale\undefined%
      \relax%
    \else%
      \setlength{\unitlength}{\unitlength * \real{\svgscale}}%
    \fi%
  \else%
    \setlength{\unitlength}{\svgwidth}%
  \fi%
  \global\let\svgwidth\undefined%
  \global\let\svgscale\undefined%
  \makeatother%
  \begin{picture}(1,0.91304107)%
    \lineheight{1}%
    \setlength\tabcolsep{0pt}%
    \put(0,0){\includegraphics[width=\unitlength,page=1]{lensplanar.pdf}}%
    \put(0.62772647,0.8260965){\color[rgb]{0,0,0}\makebox(0,0)[lt]{\lineheight{1.25}\smash{\begin{tabular}[t]{l}$\cdots$\end{tabular}}}}%
    \put(0.04091146,0.71267501){\color[rgb]{0,0,0}\makebox(0,0)[lt]{\lineheight{1.25}\smash{\begin{tabular}[t]{l}$\frac{-p}{q}$\end{tabular}}}}%
    \put(0.37776602,0.72657444){\color[rgb]{0,0,0}\makebox(0,0)[lt]{\lineheight{1.25}\smash{\begin{tabular}[t]{l}$a_1$\end{tabular}}}}%
    \put(0.51958369,0.72789961){\color[rgb]{0,0,0}\makebox(0,0)[lt]{\lineheight{1.25}\smash{\begin{tabular}[t]{l}$a_2$\end{tabular}}}}%
    \put(0.76756142,0.72421288){\color[rgb]{0,0,0}\makebox(0,0)[lt]{\lineheight{1.25}\smash{\begin{tabular}[t]{l}$a_{k-1}$\end{tabular}}}}%
    \put(0.91031897,0.72274698){\color[rgb]{0,0,0}\makebox(0,0)[lt]{\lineheight{1.25}\smash{\begin{tabular}[t]{l}$a_k$\end{tabular}}}}%
    \put(0,0){\includegraphics[width=\unitlength,page=2]{lensplanar.pdf}}%
    \put(0.52678677,0.47274577){\color[rgb]{0,0,0}\makebox(0,0)[lt]{\lineheight{1.25}\smash{\begin{tabular}[t]{l}$\vdots$\end{tabular}}}}%
    \put(0.37192283,0.34427596){\color[rgb]{0,0,0}\makebox(0,0)[lt]{\lineheight{1.25}\smash{\begin{tabular}[t]{l}$b_1$\end{tabular}}}}%
    \put(0.37406766,0.38747643){\color[rgb]{0,0,0}\makebox(0,0)[lt]{\lineheight{1.25}\smash{\begin{tabular}[t]{l}$b_2$\end{tabular}}}}%
    \put(0.37626776,0.44768768){\color[rgb]{0,0,0}\makebox(0,0)[lt]{\lineheight{1.25}\smash{\begin{tabular}[t]{l}$b_3$\end{tabular}}}}%
    \put(0.36037495,0.53422522){\color[rgb]{0,0,0}\makebox(0,0)[lt]{\lineheight{1.25}\smash{\begin{tabular}[t]{l}$b_{k-1}$\end{tabular}}}}%
    \put(0.37700686,0.57822139){\color[rgb]{0,0,0}\makebox(0,0)[lt]{\lineheight{1.25}\smash{\begin{tabular}[t]{l}$b_k$\end{tabular}}}}%
    \put(0.78247343,0.42416492){\color[rgb]{0,0,0}\makebox(0,0)[lt]{\lineheight{1.25}\smash{\begin{tabular}[t]{l}$\vdots$\end{tabular}}}}%
    \put(0.66863444,0.26534208){\color[rgb]{0,0,0}\makebox(0,0)[lt]{\lineheight{1.25}\smash{\begin{tabular}[t]{l}$a_1+1$\end{tabular}}}}%
    \put(0.70041014,0.32892229){\color[rgb]{0,0,0}\makebox(0,0)[lt]{\lineheight{1.25}\smash{\begin{tabular}[t]{l}$a_2+2$\end{tabular}}}}%
    \put(0.74655447,0.39279583){\color[rgb]{0,0,0}\makebox(0,0)[lt]{\lineheight{1.25}\smash{\begin{tabular}[t]{l}$a_3+2$\end{tabular}}}}%
    \put(0.73902307,0.48233551){\color[rgb]{0,0,0}\makebox(0,0)[lt]{\lineheight{1.25}\smash{\begin{tabular}[t]{l}$a_{k-1}+2$\end{tabular}}}}%
    \put(0.35694124,0.01042885){\color[rgb]{0,0,0}\makebox(0,0)[lt]{\lineheight{1.25}\smash{\begin{tabular}[t]{l}$b_i=\sum_{j=1}^{i}a_j +2(i-1) $\end{tabular}}}}%
    \put(0,0){\includegraphics[width=\unitlength,page=3]{lensplanar.pdf}}%
  \end{picture}%
\endgroup%

%% file: lplr.pdf_tex
\begingroup%
  \makeatletter%
  \providecommand\color[2][]{%
    \errmessage{(Inkscape) Color is used for the text in Inkscape, but the package 'color.sty' is not loaded}%
    \renewcommand\color[2][]{}%
  }%
  \providecommand\transparent[1]{%
    \errmessage{(Inkscape) Transparency is used (non-zero) for the text in Inkscape, but the package 'transparent.sty' is not loaded}%
    \renewcommand\transparent[1]{}%
  }%
  \providecommand\rotatebox[2]{#2}%
  \newcommand*\fsize{\dimexpr\f@size pt\relax}%
  \newcommand*\lineheight[1]{\fontsize{\fsize}{#1\fsize}\selectfont}%
  \ifx\svgwidth\undefined%
    \setlength{\unitlength}{570.91380172bp}%
    \ifx\svgscale\undefined%
      \relax%
    \else%
      \setlength{\unitlength}{\unitlength * \real{\svgscale}}%
    \fi%
  \else%
    \setlength{\unitlength}{\svgwidth}%
  \fi%
  \global\let\svgwidth\undefined%
  \global\let\svgscale\undefined%
  \makeatother%
  \begin{picture}(1,0.42116857)%
    \lineheight{1}%
    \setlength\tabcolsep{0pt}%
    \put(0,0){\includegraphics[width=\unitlength,page=1]{lplr.pdf}}%
    \put(0.11526977,0.12211632){\color[rgb]{0,0,0}\makebox(0,0)[lt]{\lineheight{1.25}\smash{\begin{tabular}[t]{l}$\gamma_1$\end{tabular}}}}%
    \put(0.24790706,0.11309913){\color[rgb]{0,0,0}\makebox(0,0)[lt]{\lineheight{1.25}\smash{\begin{tabular}[t]{l}$\gamma_2$\end{tabular}}}}%
    \put(0.3968689,0.11760785){\color[rgb]{0,0,0}\makebox(0,0)[lt]{\lineheight{1.25}\smash{\begin{tabular}[t]{l}$\gamma_3$\end{tabular}}}}%
    \put(0.70810635,0.1432079){\color[rgb]{0,0,0}\makebox(0,0)[lt]{\lineheight{1.25}\smash{\begin{tabular}[t]{l}$\gamma_{k-1}$\end{tabular}}}}%
    \put(0.80148049,0.1520665){\color[rgb]{0,0,0}\makebox(0,0)[lt]{\lineheight{1.25}\smash{\begin{tabular}[t]{l}$\gamma_k$\end{tabular}}}}%
    \put(0.48665408,0.0409625){\color[rgb]{0,0,0}\makebox(0,0)[lt]{\lineheight{1.25}\smash{\begin{tabular}[t]{l}$\gamma_{1,2,...,k}$\end{tabular}}}}%
    \put(0.53766848,0.19425287){\color[rgb]{0,0,0}\makebox(0,0)[lt]{\lineheight{1.25}\smash{\begin{tabular}[t]{l}$\cdots$\end{tabular}}}}%
    \put(0,0){\includegraphics[width=\unitlength,page=2]{lplr.pdf}}%
    \put(0.68207798,0.28801766){\color[rgb]{0,0,0}\makebox(0,0)[lt]{\lineheight{1.25}\smash{\begin{tabular}[t]{l}$\gamma_{k-1,k}$\end{tabular}}}}%
  \end{picture}%
\endgroup%

%% file: S03.pdf_tex
\begingroup%
  \makeatletter%
  \providecommand\color[2][]{%
    \errmessage{(Inkscape) Color is used for the text in Inkscape, but the package 'color.sty' is not loaded}%
    \renewcommand\color[2][]{}%
  }%
  \providecommand\transparent[1]{%
    \errmessage{(Inkscape) Transparency is used (non-zero) for the text in Inkscape, but the package 'transparent.sty' is not loaded}%
    \renewcommand\transparent[1]{}%
  }%
  \providecommand\rotatebox[2]{#2}%
  \newcommand*\fsize{\dimexpr\f@size pt\relax}%
  \newcommand*\lineheight[1]{\fontsize{\fsize}{#1\fsize}\selectfont}%
  \ifx\svgwidth\undefined%
    \setlength{\unitlength}{435.81353456bp}%
    \ifx\svgscale\undefined%
      \relax%
    \else%
      \setlength{\unitlength}{\unitlength * \real{\svgscale}}%
    \fi%
  \else%
    \setlength{\unitlength}{\svgwidth}%
  \fi%
  \global\let\svgwidth\undefined%
  \global\let\svgscale\undefined%
  \makeatother%
  \begin{picture}(1,0.34828805)%
    \lineheight{1}%
    \setlength\tabcolsep{0pt}%
    \put(0,0){\includegraphics[width=\unitlength,page=1]{S03.pdf}}%
    \put(0.6684081,0.09138063){\color[rgb]{0,0,0}\makebox(0,0)[lt]{\lineheight{1.25}\smash{\begin{tabular}[t]{l}$a$\end{tabular}}}}%
    \put(0.86488022,0.09464401){\color[rgb]{0,0,0}\makebox(0,0)[lt]{\lineheight{1.25}\smash{\begin{tabular}[t]{l}$c$\end{tabular}}}}%
    \put(0.77848834,0.27739149){\color[rgb]{0,0,0}\makebox(0,0)[lt]{\lineheight{1.25}\smash{\begin{tabular}[t]{l}$b$\end{tabular}}}}%
    \put(0,0){\includegraphics[width=\unitlength,page=2]{S03.pdf}}%
    \put(-0.00173329,0.1635612){\color[rgb]{0,0,0}\makebox(0,0)[lt]{\lineheight{1.25}\smash{\begin{tabular}[t]{l}$-p$\end{tabular}}}}%
    \put(0.39481498,0.16450243){\color[rgb]{0,0,0}\makebox(0,0)[lt]{\lineheight{1.25}\smash{\begin{tabular}[t]{l}$-q$\end{tabular}}}}%
  \end{picture}%
\endgroup%

%% file: Mp.pdf_tex
\begingroup%
  \makeatletter%
  \providecommand\color[2][]{%
    \errmessage{(Inkscape) Color is used for the text in Inkscape, but the package 'color.sty' is not loaded}%
    \renewcommand\color[2][]{}%
  }%
  \providecommand\transparent[1]{%
    \errmessage{(Inkscape) Transparency is used (non-zero) for the text in Inkscape, but the package 'transparent.sty' is not loaded}%
    \renewcommand\transparent[1]{}%
  }%
  \providecommand\rotatebox[2]{#2}%
  \newcommand*\fsize{\dimexpr\f@size pt\relax}%
  \newcommand*\lineheight[1]{\fontsize{\fsize}{#1\fsize}\selectfont}%
  \ifx\svgwidth\undefined%
    \setlength{\unitlength}{1617.78149308bp}%
    \ifx\svgscale\undefined%
      \relax%
    \else%
      \setlength{\unitlength}{\unitlength * \real{\svgscale}}%
    \fi%
  \else%
    \setlength{\unitlength}{\svgwidth}%
  \fi%
  \global\let\svgwidth\undefined%
  \global\let\svgscale\undefined%
  \makeatother%
  \begin{picture}(1,0.43962136)%
    \lineheight{1}%
    \setlength\tabcolsep{0pt}%
    \put(0,0){\includegraphics[width=\unitlength,page=1]{Mp.pdf}}%
    \put(0.31678631,0.20720459){\color[rgb]{0,0,0}\makebox(0,0)[lt]{\lineheight{1.25}\smash{\begin{tabular}[t]{l}$-1$\end{tabular}}}}%
    \put(0.20249858,0.16374446){\color[rgb]{0,0,0}\makebox(0,0)[lt]{\lineheight{1.25}\smash{\begin{tabular}[t]{l}$0$\end{tabular}}}}%
    \put(0.19899959,0.11844884){\color[rgb]{0,0,0}\makebox(0,0)[lt]{\lineheight{1.25}\smash{\begin{tabular}[t]{l}$0$\end{tabular}}}}%
    \put(0.19057657,0.08567531){\color[rgb]{0,0,0}\makebox(0,0)[lt]{\lineheight{1.25}\smash{\begin{tabular}[t]{l}$-3$\end{tabular}}}}%
    \put(0.18396842,0.05230758){\color[rgb]{0,0,0}\makebox(0,0)[lt]{\lineheight{1.25}\smash{\begin{tabular}[t]{l}$-3$\end{tabular}}}}%
    \put(0.17417788,0.01446223){\color[rgb]{0,0,0}\makebox(0,0)[lt]{\lineheight{1.25}\smash{\begin{tabular}[t]{l}$-p-1$\end{tabular}}}}%
    \put(0,0){\includegraphics[width=\unitlength,page=2]{Mp.pdf}}%
    \put(0.74506238,0.11719374){\color[rgb]{0,0,0}\makebox(0,0)[lt]{\lineheight{1.25}\smash{\begin{tabular}[t]{l}$a_3$\end{tabular}}}}%
    \put(0.8245203,0.12106571){\color[rgb]{0,0,0}\makebox(0,0)[lt]{\lineheight{1.25}\smash{\begin{tabular}[t]{l}$a_4$\end{tabular}}}}%
    \put(0.90897289,0.13130916){\color[rgb]{0,0,0}\makebox(0,0)[lt]{\lineheight{1.25}\smash{\begin{tabular}[t]{l}$a_5$\end{tabular}}}}%
    \put(0.75932504,0.25305588){\color[rgb]{0,0,0}\makebox(0,0)[lt]{\lineheight{1.25}\smash{\begin{tabular}[t]{l}$2$\end{tabular}}}}%
    \put(0.83656041,0.25069207){\color[rgb]{0,0,0}\makebox(0,0)[lt]{\lineheight{1.25}\smash{\begin{tabular}[t]{l}$2$\end{tabular}}}}%
    \put(0.88707247,0.25236369){\color[rgb]{0,0,0}\makebox(0,0)[lt]{\lineheight{1.25}\smash{\begin{tabular}[t]{l}$p$\end{tabular}}}}%
    \put(0.73442624,0.0438505){\color[rgb]{0,0,0}\makebox(0,0)[lt]{\lineheight{1.25}\smash{\begin{tabular}[t]{l}$b$\end{tabular}}}}%
    \put(0.81714638,0.05376555){\color[rgb]{0,0,0}\makebox(0,0)[lt]{\lineheight{1.25}\smash{\begin{tabular}[t]{l}$1$\end{tabular}}}}%
    \put(0.58515856,0.24528975){\color[rgb]{0,0,0}\makebox(0,0)[lt]{\lineheight{1.25}\smash{\begin{tabular}[t]{l}$1$\end{tabular}}}}%
    \put(0.66646225,0.25400323){\color[rgb]{0,0,0}\makebox(0,0)[lt]{\lineheight{1.25}\smash{\begin{tabular}[t]{l}$1$\end{tabular}}}}%
    \put(0,0){\includegraphics[width=\unitlength,page=3]{Mp.pdf}}%
    \put(0.60021113,0.12797892){\color[rgb]{0,0,0}\makebox(0,0)[lt]{\lineheight{1.25}\smash{\begin{tabular}[t]{l}$a_1$\end{tabular}}}}%
    \put(0.67660167,0.13293934){\color[rgb]{0,0,0}\makebox(0,0)[lt]{\lineheight{1.25}\smash{\begin{tabular}[t]{l}$a_2$\end{tabular}}}}%
  \end{picture}%
\endgroup%

%% file: tripodplumb.pdf_tex
\begingroup%
  \makeatletter%
  \providecommand\color[2][]{%
    \errmessage{(Inkscape) Color is used for the text in Inkscape, but the package 'color.sty' is not loaded}%
    \renewcommand\color[2][]{}%
  }%
  \providecommand\transparent[1]{%
    \errmessage{(Inkscape) Transparency is used (non-zero) for the text in Inkscape, but the package 'transparent.sty' is not loaded}%
    \renewcommand\transparent[1]{}%
  }%
  \providecommand\rotatebox[2]{#2}%
  \newcommand*\fsize{\dimexpr\f@size pt\relax}%
  \newcommand*\lineheight[1]{\fontsize{\fsize}{#1\fsize}\selectfont}%
  \ifx\svgwidth\undefined%
    \setlength{\unitlength}{430.50329013bp}%
    \ifx\svgscale\undefined%
      \relax%
    \else%
      \setlength{\unitlength}{\unitlength * \real{\svgscale}}%
    \fi%
  \else%
    \setlength{\unitlength}{\svgwidth}%
  \fi%
  \global\let\svgwidth\undefined%
  \global\let\svgscale\undefined%
  \makeatother%
  \begin{picture}(1,0.43897931)%
    \lineheight{1}%
    \setlength\tabcolsep{0pt}%
    \put(0,0){\includegraphics[width=\unitlength,page=1]{tripodplumb.pdf}}%
    \put(-0.00085793,0.42216738){\color[rgb]{0.03921569,0,0.00784314}\makebox(0,0)[lt]{\lineheight{1.25}\smash{\begin{tabular}[t]{l}$-3$\end{tabular}}}}%
    \put(0.26087272,0.42204104){\color[rgb]{0.03921569,0,0.00784314}\makebox(0,0)[lt]{\lineheight{1.25}\smash{\begin{tabular}[t]{l}$-2$\end{tabular}}}}%
    \put(0.47676151,0.42204104){\color[rgb]{0.03921569,0,0.00784314}\makebox(0,0)[lt]{\lineheight{1.25}\smash{\begin{tabular}[t]{l}$-3$\end{tabular}}}}%
    \put(0.73011858,0.42204104){\color[rgb]{0.03921569,0,0.00784314}\makebox(0,0)[lt]{\lineheight{1.25}\smash{\begin{tabular}[t]{l}$-2$\end{tabular}}}}%
    \put(0.95135988,0.41847255){\color[rgb]{0.03921569,0,0.00784314}\makebox(0,0)[lt]{\lineheight{1.25}\smash{\begin{tabular}[t]{l}$-3$\end{tabular}}}}%
    \put(0.53028763,0.19366286){\color[rgb]{0.03921569,0,0.00784314}\makebox(0,0)[lt]{\lineheight{1.25}\smash{\begin{tabular}[t]{l}$-2$\end{tabular}}}}%
    \put(0.53207185,0.00632131){\color[rgb]{0.03921569,0,0.00784314}\makebox(0,0)[lt]{\lineheight{1.25}\smash{\begin{tabular}[t]{l}$-2$\end{tabular}}}}%
  \end{picture}%
\endgroup%

%% file: ex.pdf_tex
\begingroup%
  \makeatletter%
  \providecommand\color[2][]{%
    \errmessage{(Inkscape) Color is used for the text in Inkscape, but the package 'color.sty' is not loaded}%
    \renewcommand\color[2][]{}%
  }%
  \providecommand\transparent[1]{%
    \errmessage{(Inkscape) Transparency is used (non-zero) for the text in Inkscape, but the package 'transparent.sty' is not loaded}%
    \renewcommand\transparent[1]{}%
  }%
  \providecommand\rotatebox[2]{#2}%
  \newcommand*\fsize{\dimexpr\f@size pt\relax}%
  \newcommand*\lineheight[1]{\fontsize{\fsize}{#1\fsize}\selectfont}%
  \ifx\svgwidth\undefined%
    \setlength{\unitlength}{1181.82088715bp}%
    \ifx\svgscale\undefined%
      \relax%
    \else%
      \setlength{\unitlength}{\unitlength * \real{\svgscale}}%
    \fi%
  \else%
    \setlength{\unitlength}{\svgwidth}%
  \fi%
  \global\let\svgwidth\undefined%
  \global\let\svgscale\undefined%
  \makeatother%
  \begin{picture}(1,0.40640394)%
    \lineheight{1}%
    \setlength\tabcolsep{0pt}%
    \put(0,0){\includegraphics[width=\unitlength,page=1]{ex.pdf}}%
    \put(0.23582258,0.25760142){\color[rgb]{0,0,0}\makebox(0,0)[lt]{\lineheight{1.25}\smash{\begin{tabular}[t]{l}$-3$\end{tabular}}}}%
    \put(0.27614817,0.1975748){\color[rgb]{0,0,0}\makebox(0,0)[lt]{\lineheight{1.25}\smash{\begin{tabular}[t]{l}$-2$\end{tabular}}}}%
    \put(0.31150689,0.14081208){\color[rgb]{0,0,0}\makebox(0,0)[lt]{\lineheight{1.25}\smash{\begin{tabular}[t]{l}$2$\end{tabular}}}}%
    \put(0.29023121,0.07198002){\color[rgb]{0,0,0}\makebox(0,0)[lt]{\lineheight{1.25}\smash{\begin{tabular}[t]{l}$-1$\end{tabular}}}}%
    \put(0.1843361,0.39549669){\color[rgb]{0,0,0}\makebox(0,0)[lt]{\lineheight{1.25}\smash{\begin{tabular}[t]{l}$-3$\end{tabular}}}}%
    \put(0.1811107,0.36081682){\color[rgb]{0,0,0}\makebox(0,0)[lt]{\lineheight{1.25}\smash{\begin{tabular}[t]{l}$-2$\end{tabular}}}}%
    \put(0.17461472,0.33380612){\color[rgb]{0,0,0}\makebox(0,0)[lt]{\lineheight{1.25}\smash{\begin{tabular}[t]{l}$-4$\end{tabular}}}}%
    \put(0.16166098,0.2404927){\color[rgb]{0,0,0}\makebox(0,0)[lt]{\lineheight{1.25}\smash{\begin{tabular}[t]{l}$-3$\end{tabular}}}}%
    \put(0.17231998,0.27467022){\color[rgb]{0,0,0}\makebox(0,0)[lt]{\lineheight{1.25}\smash{\begin{tabular}[t]{l}$-3$\end{tabular}}}}%
    \put(0.16968558,0.30322645){\color[rgb]{0,0,0}\makebox(0,0)[lt]{\lineheight{1.25}\smash{\begin{tabular}[t]{l}$-2$\end{tabular}}}}%
    \put(0,0){\includegraphics[width=\unitlength,page=2]{ex.pdf}}%
    \put(0.56734881,0.10259541){\color[rgb]{0,0,0}\makebox(0,0)[lt]{\lineheight{1.25}\smash{\begin{tabular}[t]{l}$2$\end{tabular}}}}%
    \put(0.47886029,0.16320325){\color[rgb]{0,0,0}\makebox(0,0)[lt]{\lineheight{1.25}\smash{\begin{tabular}[t]{l}$1$\end{tabular}}}}%
    \put(0.56873589,0.17085679){\color[rgb]{0,0,0}\makebox(0,0)[lt]{\lineheight{1.25}\smash{\begin{tabular}[t]{l}$-2$\end{tabular}}}}%
    \put(0.62959284,0.17321665){\color[rgb]{0,0,0}\makebox(0,0)[lt]{\lineheight{1.25}\smash{\begin{tabular}[t]{l}$-1$\end{tabular}}}}%
    \put(0.73386741,0.16416978){\color[rgb]{0,0,0}\makebox(0,0)[lt]{\lineheight{1.25}\smash{\begin{tabular}[t]{l}$3$\end{tabular}}}}%
    \put(0.8197146,0.16275964){\color[rgb]{0,0,0}\makebox(0,0)[lt]{\lineheight{1.25}\smash{\begin{tabular}[t]{l}$1$\end{tabular}}}}%
    \put(0.88994765,0.15962595){\color[rgb]{0,0,0}\makebox(0,0)[lt]{\lineheight{1.25}\smash{\begin{tabular}[t]{l}$2$\end{tabular}}}}%
    \put(0.60443077,0.14453253){\color[rgb]{0,0,0}\makebox(0,0)[lt]{\lineheight{1.25}\smash{\begin{tabular}[t]{l}$3$\end{tabular}}}}%
    \put(0.78705395,0.14194064){\color[rgb]{0,0,0}\makebox(0,0)[lt]{\lineheight{1.25}\smash{\begin{tabular}[t]{l}$1$\end{tabular}}}}%
    \put(0,0){\includegraphics[width=\unitlength,page=3]{ex.pdf}}%
    \put(0.7227698,0.09485306){\color[rgb]{0,0,0}\makebox(0,0)[lt]{\lineheight{1.25}\smash{\begin{tabular}[t]{l}$-2$\end{tabular}}}}%
    \put(0.48213813,0.24622228){\color[rgb]{0,0,0}\makebox(0,0)[lt]{\lineheight{1.25}\smash{\begin{tabular}[t]{l}$a_1$\end{tabular}}}}%
    \put(0.5605664,0.2486868){\color[rgb]{0,0,0}\makebox(0,0)[lt]{\lineheight{1.25}\smash{\begin{tabular}[t]{l}$a_2$\end{tabular}}}}%
    \put(0.66891789,0.23908888){\color[rgb]{0,0,0}\makebox(0,0)[lt]{\lineheight{1.25}\smash{\begin{tabular}[t]{l}$a_3$\end{tabular}}}}%
    \put(0.74089047,0.23847276){\color[rgb]{0,0,0}\makebox(0,0)[lt]{\lineheight{1.25}\smash{\begin{tabular}[t]{l}$a_4$\end{tabular}}}}%
    \put(0.81098738,0.24023901){\color[rgb]{0,0,0}\makebox(0,0)[lt]{\lineheight{1.25}\smash{\begin{tabular}[t]{l}$a_5$\end{tabular}}}}%
    \put(0.88041538,0.23741849){\color[rgb]{0,0,0}\makebox(0,0)[lt]{\lineheight{1.25}\smash{\begin{tabular}[t]{l}$a_6$\end{tabular}}}}%
    \put(0.69450081,0.26206368){\color[rgb]{0,0,0}\makebox(0,0)[lt]{\lineheight{1.25}\smash{\begin{tabular}[t]{l}$b_1$\end{tabular}}}}%
    \put(0.78104253,0.27289389){\color[rgb]{0,0,0}\makebox(0,0)[lt]{\lineheight{1.25}\smash{\begin{tabular}[t]{l}$b_2$\end{tabular}}}}%
    \put(0.68552969,0.29567701){\color[rgb]{0,0,0}\makebox(0,0)[lt]{\lineheight{1.25}\smash{\begin{tabular}[t]{l}$b_3$\end{tabular}}}}%
    \put(0.78038802,0.0672297){\color[rgb]{0,0,0}\makebox(0,0)[lt]{\lineheight{1.25}\smash{\begin{tabular}[t]{l}$b_4$\end{tabular}}}}%
  \end{picture}%
\endgroup%

%% file: phsphere.pdf_tex
\begingroup%
  \makeatletter%
  \providecommand\color[2][]{%
    \errmessage{(Inkscape) Color is used for the text in Inkscape, but the package 'color.sty' is not loaded}%
    \renewcommand\color[2][]{}%
  }%
  \providecommand\transparent[1]{%
    \errmessage{(Inkscape) Transparency is used (non-zero) for the text in Inkscape, but the package 'transparent.sty' is not loaded}%
    \renewcommand\transparent[1]{}%
  }%
  \providecommand\rotatebox[2]{#2}%
  \newcommand*\fsize{\dimexpr\f@size pt\relax}%
  \newcommand*\lineheight[1]{\fontsize{\fsize}{#1\fsize}\selectfont}%
  \ifx\svgwidth\undefined%
    \setlength{\unitlength}{430.46048461bp}%
    \ifx\svgscale\undefined%
      \relax%
    \else%
      \setlength{\unitlength}{\unitlength * \real{\svgscale}}%
    \fi%
  \else%
    \setlength{\unitlength}{\svgwidth}%
  \fi%
  \global\let\svgwidth\undefined%
  \global\let\svgscale\undefined%
  \makeatother%
  \begin{picture}(1,0.60199894)%
    \lineheight{1}%
    \setlength\tabcolsep{0pt}%
    \put(0,0){\includegraphics[width=\unitlength,page=1]{phsphere.pdf}}%
    \put(0.27551122,0.22288899){\color[rgb]{0,0,0}\makebox(0,0)[lt]{\lineheight{1.25}\smash{\begin{tabular}[t]{l}$\delta_3$\end{tabular}}}}%
    \put(0.69712695,0.22610528){\color[rgb]{0,0,0}\makebox(0,0)[lt]{\lineheight{1.25}\smash{\begin{tabular}[t]{l}$\delta_1$\end{tabular}}}}%
    \put(0.5063186,0.41859319){\color[rgb]{0,0,0}\makebox(0,0)[lt]{\lineheight{1.25}\smash{\begin{tabular}[t]{l}$\delta_2$\end{tabular}}}}%
    \put(0.18459827,0.24413921){\color[rgb]{0,0,0}\makebox(0,0)[lt]{\lineheight{1.25}\smash{\begin{tabular}[t]{l}$-$\end{tabular}}}}%
    \put(0.7884831,0.20349979){\color[rgb]{0,0,0}\makebox(0,0)[lt]{\lineheight{1.25}\smash{\begin{tabular}[t]{l}$-$\end{tabular}}}}%
    \put(0.50907551,0.331633){\color[rgb]{0,0,0}\makebox(0,0)[lt]{\lineheight{1.25}\smash{\begin{tabular}[t]{l}$+$\end{tabular}}}}%
    \put(0,0){\includegraphics[width=\unitlength,page=2]{phsphere.pdf}}%
    \put(0.03673612,0.19940228){\color[rgb]{0,0,0}\makebox(0,0)[lt]{\lineheight{1.25}\smash{\begin{tabular}[t]{l}$+$\end{tabular}}}}%
    \put(0.91792724,0.25688995){\color[rgb]{0,0,0}\makebox(0,0)[lt]{\lineheight{1.25}\smash{\begin{tabular}[t]{l}$+$\end{tabular}}}}%
    \put(0.82803596,0.13856101){\color[rgb]{0,0,0}\makebox(0,0)[lt]{\lineheight{1.25}\smash{\begin{tabular}[t]{l}$-$\end{tabular}}}}%
    \put(0.16845571,0.42183805){\color[rgb]{0,0,0}\makebox(0,0)[lt]{\lineheight{1.25}\smash{\begin{tabular}[t]{l}$\beta$\end{tabular}}}}%
    \put(0.8373387,0.4129301){\color[rgb]{0,0,0}\makebox(0,0)[lt]{\lineheight{1.25}\smash{\begin{tabular}[t]{l}$\alpha$\end{tabular}}}}%
    \put(0.11398028,0.20508948){\color[rgb]{0,0,0}\makebox(0,0)[lt]{\lineheight{1.25}\smash{\begin{tabular}[t]{l}$-$\end{tabular}}}}%
  \end{picture}%
\endgroup%

%% file: ph0.pdf_tex
\begingroup%
  \makeatletter%
  \providecommand\color[2][]{%
    \errmessage{(Inkscape) Color is used for the text in Inkscape, but the package 'color.sty' is not loaded}%
    \renewcommand\color[2][]{}%
  }%
  \providecommand\transparent[1]{%
    \errmessage{(Inkscape) Transparency is used (non-zero) for the text in Inkscape, but the package 'transparent.sty' is not loaded}%
    \renewcommand\transparent[1]{}%
  }%
  \providecommand\rotatebox[2]{#2}%
  \newcommand*\fsize{\dimexpr\f@size pt\relax}%
  \newcommand*\lineheight[1]{\fontsize{\fsize}{#1\fsize}\selectfont}%
  \ifx\svgwidth\undefined%
    \setlength{\unitlength}{315.20697304bp}%
    \ifx\svgscale\undefined%
      \relax%
    \else%
      \setlength{\unitlength}{\unitlength * \real{\svgscale}}%
    \fi%
  \else%
    \setlength{\unitlength}{\svgwidth}%
  \fi%
  \global\let\svgwidth\undefined%
  \global\let\svgscale\undefined%
  \makeatother%
  \begin{picture}(1,0.23894614)%
    \lineheight{1}%
    \setlength\tabcolsep{0pt}%
    \put(0,0){\includegraphics[width=\unitlength,page=1]{ph0.pdf}}%
    \put(-0.00316013,0.1882387){\color[rgb]{0,0,0}\makebox(0,0)[lt]{\lineheight{1.25}\smash{\begin{tabular}[t]{l}$-2$\end{tabular}}}}%
    \put(0.11580933,0.19661677){\color[rgb]{0,0,0}\makebox(0,0)[lt]{\lineheight{1.25}\smash{\begin{tabular}[t]{l}$-2$\end{tabular}}}}%
    \put(0.29677693,0.19829239){\color[rgb]{0,0,0}\makebox(0,0)[lt]{\lineheight{1.25}\smash{\begin{tabular}[t]{l}$-2$\end{tabular}}}}%
    \put(0.46601515,0.19829239){\color[rgb]{0,0,0}\makebox(0,0)[lt]{\lineheight{1.25}\smash{\begin{tabular}[t]{l}$-2$\end{tabular}}}}%
    \put(0.61011901,0.199968){\color[rgb]{0,0,0}\makebox(0,0)[lt]{\lineheight{1.25}\smash{\begin{tabular}[t]{l}$-2$\end{tabular}}}}%
    \put(0.76260102,0.20667047){\color[rgb]{0,0,0}\makebox(0,0)[lt]{\lineheight{1.25}\smash{\begin{tabular}[t]{l}$-2$\end{tabular}}}}%
    \put(0.91675856,0.2100217){\color[rgb]{0,0,0}\makebox(0,0)[lt]{\lineheight{1.25}\smash{\begin{tabular}[t]{l}$-2$\end{tabular}}}}%
    \put(0.61849723,0.00559528){\color[rgb]{0,0,0}\makebox(0,0)[lt]{\lineheight{1.25}\smash{\begin{tabular}[t]{l}$-2$\end{tabular}}}}%
  \end{picture}%
\endgroup%

%% file: ph1.pdf_tex
\begingroup%
  \makeatletter%
  \providecommand\color[2][]{%
    \errmessage{(Inkscape) Color is used for the text in Inkscape, but the package 'color.sty' is not loaded}%
    \renewcommand\color[2][]{}%
  }%
  \providecommand\transparent[1]{%
    \errmessage{(Inkscape) Transparency is used (non-zero) for the text in Inkscape, but the package 'transparent.sty' is not loaded}%
    \renewcommand\transparent[1]{}%
  }%
  \providecommand\rotatebox[2]{#2}%
  \newcommand*\fsize{\dimexpr\f@size pt\relax}%
  \newcommand*\lineheight[1]{\fontsize{\fsize}{#1\fsize}\selectfont}%
  \ifx\svgwidth\undefined%
    \setlength{\unitlength}{3489.2718601bp}%
    \ifx\svgscale\undefined%
      \relax%
    \else%
      \setlength{\unitlength}{\unitlength * \real{\svgscale}}%
    \fi%
  \else%
    \setlength{\unitlength}{\svgwidth}%
  \fi%
  \global\let\svgwidth\undefined%
  \global\let\svgscale\undefined%
  \makeatother%
  \begin{picture}(1,0.57936734)%
    \lineheight{1}%
    \setlength\tabcolsep{0pt}%
    \put(0,0){\includegraphics[width=\unitlength,page=1]{ph1.pdf}}%
    \put(-0.00074819,0.29186446){\color[rgb]{0,0,0}\makebox(0,0)[lt]{\lineheight{1.25}\smash{\begin{tabular}[t]{l}$-2$\end{tabular}}}}%
    \put(0.2234346,0.24990277){\color[rgb]{0,0,0}\makebox(0,0)[lt]{\lineheight{1.25}\smash{\begin{tabular}[t]{l}$-2$\end{tabular}}}}%
    \put(0.21198575,0.21096032){\color[rgb]{0,0,0}\makebox(0,0)[lt]{\lineheight{1.25}\smash{\begin{tabular}[t]{l}$-2$\end{tabular}}}}%
    \put(0.23286052,0.05560312){\color[rgb]{0,0,0}\makebox(0,0)[lt]{\lineheight{1.25}\smash{\begin{tabular}[t]{l}$-2$\end{tabular}}}}%
    \put(0.2354943,0.00780182){\color[rgb]{0,0,0}\makebox(0,0)[lt]{\lineheight{1.25}\smash{\begin{tabular}[t]{l}$-2$\end{tabular}}}}%
    \put(0.23022681,0.10340443){\color[rgb]{0,0,0}\makebox(0,0)[lt]{\lineheight{1.25}\smash{\begin{tabular}[t]{l}$-2$\end{tabular}}}}%
    \put(0.22076855,0.14698796){\color[rgb]{0,0,0}\makebox(0,0)[lt]{\lineheight{1.25}\smash{\begin{tabular}[t]{l}$-2$\end{tabular}}}}%
    \put(0.22370029,0.18051473){\color[rgb]{0,0,0}\makebox(0,0)[lt]{\lineheight{1.25}\smash{\begin{tabular}[t]{l}$-2$\end{tabular}}}}%
    \put(0,0){\includegraphics[width=\unitlength,page=2]{ph1.pdf}}%
    \put(0.31992421,0.29202105){\color[rgb]{0,0,0}\makebox(0,0)[lt]{\lineheight{1.25}\smash{\begin{tabular}[t]{l}$-1$\end{tabular}}}}%
    \put(0,0){\includegraphics[width=\unitlength,page=3]{ph1.pdf}}%
    \put(0.73747718,0.57126993){\color[rgb]{0,0,0}\makebox(0,0)[lt]{\lineheight{1.25}\smash{\begin{tabular}[t]{l}$-1$\end{tabular}}}}%
    \put(0.74444678,0.53808805){\color[rgb]{0,0,0}\makebox(0,0)[lt]{\lineheight{1.25}\smash{\begin{tabular}[t]{l}$-1$\end{tabular}}}}%
    \put(0.74571392,0.49947655){\color[rgb]{0,0,0}\makebox(0,0)[lt]{\lineheight{1.25}\smash{\begin{tabular}[t]{l}$-1$\end{tabular}}}}%
    \put(0.74571392,0.45724514){\color[rgb]{0,0,0}\makebox(0,0)[lt]{\lineheight{1.25}\smash{\begin{tabular}[t]{l}$-1$\end{tabular}}}}%
    \put(0.75014904,0.42044343){\color[rgb]{0,0,0}\makebox(0,0)[lt]{\lineheight{1.25}\smash{\begin{tabular}[t]{l}$-1$\end{tabular}}}}%
    \put(0.74216652,0.38889095){\color[rgb]{0,0,0}\makebox(0,0)[lt]{\lineheight{1.25}\smash{\begin{tabular}[t]{l}$-1$\end{tabular}}}}%
    \put(0.7429898,0.36025377){\color[rgb]{0,0,0}\makebox(0,0)[lt]{\lineheight{1.25}\smash{\begin{tabular}[t]{l}$-1$\end{tabular}}}}%
    \put(0.74481655,0.33253243){\color[rgb]{0,0,0}\makebox(0,0)[lt]{\lineheight{1.25}\smash{\begin{tabular}[t]{l}$-1$\end{tabular}}}}%
    \put(0.96824188,0.27405929){\color[rgb]{0,0,0}\makebox(0,0)[lt]{\lineheight{1.25}\smash{\begin{tabular}[t]{l}$1$\end{tabular}}}}%
  \end{picture}%
\endgroup%

%% file: ph2.pdf_tex
\begingroup%
  \makeatletter%
  \providecommand\color[2][]{%
    \errmessage{(Inkscape) Color is used for the text in Inkscape, but the package 'color.sty' is not loaded}%
    \renewcommand\color[2][]{}%
  }%
  \providecommand\transparent[1]{%
    \errmessage{(Inkscape) Transparency is used (non-zero) for the text in Inkscape, but the package 'transparent.sty' is not loaded}%
    \renewcommand\transparent[1]{}%
  }%
  \providecommand\rotatebox[2]{#2}%
  \newcommand*\fsize{\dimexpr\f@size pt\relax}%
  \newcommand*\lineheight[1]{\fontsize{\fsize}{#1\fsize}\selectfont}%
  \ifx\svgwidth\undefined%
    \setlength{\unitlength}{4018.38853905bp}%
    \ifx\svgscale\undefined%
      \relax%
    \else%
      \setlength{\unitlength}{\unitlength * \real{\svgscale}}%
    \fi%
  \else%
    \setlength{\unitlength}{\svgwidth}%
  \fi%
  \global\let\svgwidth\undefined%
  \global\let\svgscale\undefined%
  \makeatother%
  \begin{picture}(1,0.5938163)%
    \lineheight{1}%
    \setlength\tabcolsep{0pt}%
    \put(0,0){\includegraphics[width=\unitlength,page=1]{ph2.pdf}}%
    \put(0.51895234,0.31040607){\color[rgb]{0,0,0}\makebox(0,0)[lt]{\lineheight{1.25}\smash{\begin{tabular}[t]{l}$a_1$\end{tabular}}}}%
    \put(0.56967132,0.30858484){\color[rgb]{0,0,0}\makebox(0,0)[lt]{\lineheight{1.25}\smash{\begin{tabular}[t]{l}$a_2$\end{tabular}}}}%
    \put(0.6333949,0.30494234){\color[rgb]{0,0,0}\makebox(0,0)[lt]{\lineheight{1.25}\smash{\begin{tabular}[t]{l}$a_3$\end{tabular}}}}%
    \put(0.69051074,0.30613219){\color[rgb]{0,0,0}\makebox(0,0)[lt]{\lineheight{1.25}\smash{\begin{tabular}[t]{l}$a_4$\end{tabular}}}}%
    \put(0.77927818,0.31235358){\color[rgb]{0,0,0}\makebox(0,0)[lt]{\lineheight{1.25}\smash{\begin{tabular}[t]{l}$a_5$\end{tabular}}}}%
    \put(0.82354674,0.31533798){\color[rgb]{0,0,0}\makebox(0,0)[lt]{\lineheight{1.25}\smash{\begin{tabular}[t]{l}$a_6$\end{tabular}}}}%
    \put(0.87828078,0.31419672){\color[rgb]{0,0,0}\makebox(0,0)[lt]{\lineheight{1.25}\smash{\begin{tabular}[t]{l}$a_7$\end{tabular}}}}%
    \put(0.93940366,0.31237553){\color[rgb]{0,0,0}\makebox(0,0)[lt]{\lineheight{1.25}\smash{\begin{tabular}[t]{l}$a_8$\end{tabular}}}}%
    \put(0.7575234,0.24245036){\color[rgb]{0,0,0}\makebox(0,0)[lt]{\lineheight{1.25}\smash{\begin{tabular}[t]{l}$b_1$\end{tabular}}}}%
    \put(0.78446053,0.2768447){\color[rgb]{0,0,0}\makebox(0,0)[lt]{\lineheight{1.25}\smash{\begin{tabular}[t]{l}$b_2$\end{tabular}}}}%
    \put(0.8683959,0.28192485){\color[rgb]{0,0,0}\makebox(0,0)[lt]{\lineheight{1.25}\smash{\begin{tabular}[t]{l}$b_3$\end{tabular}}}}%
    \put(0.52415431,0.37961332){\color[rgb]{0,0,0}\makebox(0,0)[lt]{\lineheight{1.25}\smash{\begin{tabular}[t]{l}$1$\end{tabular}}}}%
    \put(0.56316888,0.38416638){\color[rgb]{0,0,0}\makebox(0,0)[lt]{\lineheight{1.25}\smash{\begin{tabular}[t]{l}$1$\end{tabular}}}}%
    \put(0.62689265,0.38325578){\color[rgb]{0,0,0}\makebox(0,0)[lt]{\lineheight{1.25}\smash{\begin{tabular}[t]{l}$1$\end{tabular}}}}%
    \put(0.69126647,0.38689825){\color[rgb]{0,0,0}\makebox(0,0)[lt]{\lineheight{1.25}\smash{\begin{tabular}[t]{l}$1$\end{tabular}}}}%
    \put(0.778399,0.39964699){\color[rgb]{0,0,0}\makebox(0,0)[lt]{\lineheight{1.25}\smash{\begin{tabular}[t]{l}$-1$\end{tabular}}}}%
    \put(0.8518763,0.40055762){\color[rgb]{0,0,0}\makebox(0,0)[lt]{\lineheight{1.25}\smash{\begin{tabular}[t]{l}$-1$\end{tabular}}}}%
    \put(0.72507914,0.43880369){\color[rgb]{0,0,0}\makebox(0,0)[lt]{\lineheight{1.25}\smash{\begin{tabular}[t]{l}$1$\end{tabular}}}}%
    \put(0.76189841,0.36576308){\color[rgb]{0,0,0}\makebox(0,0)[lt]{\lineheight{1.25}\smash{\begin{tabular}[t]{l}$2$\end{tabular}}}}%
    \put(0.82346872,0.36529706){\color[rgb]{0,0,0}\makebox(0,0)[lt]{\lineheight{1.25}\smash{\begin{tabular}[t]{l}$3$\end{tabular}}}}%
    \put(0.88744785,0.36713968){\color[rgb]{0,0,0}\makebox(0,0)[lt]{\lineheight{1.25}\smash{\begin{tabular}[t]{l}$2$\end{tabular}}}}%
    \put(0.94486107,0.37414959){\color[rgb]{0,0,0}\makebox(0,0)[lt]{\lineheight{1.25}\smash{\begin{tabular}[t]{l}$1$\end{tabular}}}}%
    \put(0,0){\includegraphics[width=\unitlength,page=2]{ph2.pdf}}%
    \put(0.35021605,0.27805484){\color[rgb]{0,0,0}\makebox(0,0)[lt]{\lineheight{1.25}\smash{\begin{tabular}[t]{l}$1$\end{tabular}}}}%
    \put(0.15130569,0.00786716){\color[rgb]{0,0,0}\makebox(0,0)[lt]{\lineheight{1.25}\smash{\begin{tabular}[t]{l}$-1$\end{tabular}}}}%
    \put(0.14577788,0.04752588){\color[rgb]{0,0,0}\makebox(0,0)[lt]{\lineheight{1.25}\smash{\begin{tabular}[t]{l}$-1$\end{tabular}}}}%
    \put(0.15012561,0.09520219){\color[rgb]{0,0,0}\makebox(0,0)[lt]{\lineheight{1.25}\smash{\begin{tabular}[t]{l}$-1$\end{tabular}}}}%
    \put(0.15019842,0.12242209){\color[rgb]{0,0,0}\makebox(0,0)[lt]{\lineheight{1.25}\smash{\begin{tabular}[t]{l}$-1$\end{tabular}}}}%
    \put(0.14399738,0.16225099){\color[rgb]{0,0,0}\makebox(0,0)[lt]{\lineheight{1.25}\smash{\begin{tabular}[t]{l}$-1$\end{tabular}}}}%
    \put(0.14402272,0.19884317){\color[rgb]{0,0,0}\makebox(0,0)[lt]{\lineheight{1.25}\smash{\begin{tabular}[t]{l}$-1$\end{tabular}}}}%
    \put(0.14314835,0.23124436){\color[rgb]{0,0,0}\makebox(0,0)[lt]{\lineheight{1.25}\smash{\begin{tabular}[t]{l}$-1$\end{tabular}}}}%
    \put(0.1490568,0.26963298){\color[rgb]{0,0,0}\makebox(0,0)[lt]{\lineheight{1.25}\smash{\begin{tabular}[t]{l}$-1$\end{tabular}}}}%
  \end{picture}%
\endgroup%

%% file: lek.pdf_tex
\begingroup%
  \makeatletter%
  \providecommand\color[2][]{%
    \errmessage{(Inkscape) Color is used for the text in Inkscape, but the package 'color.sty' is not loaded}%
    \renewcommand\color[2][]{}%
  }%
  \providecommand\transparent[1]{%
    \errmessage{(Inkscape) Transparency is used (non-zero) for the text in Inkscape, but the package 'transparent.sty' is not loaded}%
    \renewcommand\transparent[1]{}%
  }%
  \providecommand\rotatebox[2]{#2}%
  \newcommand*\fsize{\dimexpr\f@size pt\relax}%
  \newcommand*\lineheight[1]{\fontsize{\fsize}{#1\fsize}\selectfont}%
  \ifx\svgwidth\undefined%
    \setlength{\unitlength}{430.46050983bp}%
    \ifx\svgscale\undefined%
      \relax%
    \else%
      \setlength{\unitlength}{\unitlength * \real{\svgscale}}%
    \fi%
  \else%
    \setlength{\unitlength}{\svgwidth}%
  \fi%
  \global\let\svgwidth\undefined%
  \global\let\svgscale\undefined%
  \makeatother%
  \begin{picture}(1,0.60199895)%
    \lineheight{1}%
    \setlength\tabcolsep{0pt}%
    \put(0,0){\includegraphics[width=\unitlength,page=1]{lek.pdf}}%
    \put(0.41738784,0.42094416){\color[rgb]{0,0,0}\makebox(0,0)[lt]{\lineheight{1.25}\smash{\begin{tabular}[t]{l}$a$\end{tabular}}}}%
    \put(0.36415383,0.22534476){\color[rgb]{0,0,0}\makebox(0,0)[lt]{\lineheight{1.25}\smash{\begin{tabular}[t]{l}$b$\end{tabular}}}}%
    \put(0.60669076,0.22437704){\color[rgb]{0,0,0}\makebox(0,0)[lt]{\lineheight{1.25}\smash{\begin{tabular}[t]{l}$c$\end{tabular}}}}%
    \put(0.47309383,0.56009566){\color[rgb]{0,0,0}\makebox(0,0)[lt]{\lineheight{1.25}\smash{\begin{tabular}[t]{l}$d$\end{tabular}}}}%
    \put(0.22588714,0.39201493){\color[rgb]{0,0,0}\makebox(0,0)[lt]{\lineheight{1.25}\smash{\begin{tabular}[t]{l}$e$\end{tabular}}}}%
    \put(0.48917397,0.08782971){\color[rgb]{0,0,0}\makebox(0,0)[lt]{\lineheight{1.25}\smash{\begin{tabular}[t]{l}$f$\end{tabular}}}}%
    \put(0.77505533,0.3791697){\color[rgb]{0,0,0}\makebox(0,0)[lt]{\lineheight{1.25}\smash{\begin{tabular}[t]{l}$g$\end{tabular}}}}%
  \end{picture}%
\endgroup%

%% file: FIGST.pdf_tex
\begingroup%
  \makeatletter%
  \providecommand\color[2][]{%
    \errmessage{(Inkscape) Color is used for the text in Inkscape, but the package 'color.sty' is not loaded}%
    \renewcommand\color[2][]{}%
  }%
  \providecommand\transparent[1]{%
    \errmessage{(Inkscape) Transparency is used (non-zero) for the text in Inkscape, but the package 'transparent.sty' is not loaded}%
    \renewcommand\transparent[1]{}%
  }%
  \providecommand\rotatebox[2]{#2}%
  \newcommand*\fsize{\dimexpr\f@size pt\relax}%
  \newcommand*\lineheight[1]{\fontsize{\fsize}{#1\fsize}\selectfont}%
  \ifx\svgwidth\undefined%
    \setlength{\unitlength}{714.33073165bp}%
    \ifx\svgscale\undefined%
      \relax%
    \else%
      \setlength{\unitlength}{\unitlength * \real{\svgscale}}%
    \fi%
  \else%
    \setlength{\unitlength}{\svgwidth}%
  \fi%
  \global\let\svgwidth\undefined%
  \global\let\svgscale\undefined%
  \makeatother%
  \begin{picture}(1,0.44103136)%
    \lineheight{1}%
    \setlength\tabcolsep{0pt}%
    \put(0,0){\includegraphics[width=\unitlength,page=1]{FIGST.pdf}}%
    \put(0.65714615,0.20057292){\color[rgb]{0,0,0}\makebox(0,0)[lt]{\lineheight{1.25}\smash{\begin{tabular}[t]{l}$c$\end{tabular}}}}%
    \put(0.45075832,0.19947762){\color[rgb]{0,0,0}\makebox(0,0)[lt]{\lineheight{1.25}\smash{\begin{tabular}[t]{l}$b$\end{tabular}}}}%
    \put(0,0){\includegraphics[width=\unitlength,page=2]{FIGST.pdf}}%
    \put(0.36282049,0.26120626){\color[rgb]{0,0,0}\makebox(0,0)[lt]{\lineheight{1.25}\smash{\begin{tabular}[t]{l}$a$\end{tabular}}}}%
  \end{picture}%
\endgroup%

%% file: embeds4.pdf_tex
\begingroup%
  \makeatletter%
  \providecommand\color[2][]{%
    \errmessage{(Inkscape) Color is used for the text in Inkscape, but the package 'color.sty' is not loaded}%
    \renewcommand\color[2][]{}%
  }%
  \providecommand\transparent[1]{%
    \errmessage{(Inkscape) Transparency is used (non-zero) for the text in Inkscape, but the package 'transparent.sty' is not loaded}%
    \renewcommand\transparent[1]{}%
  }%
  \providecommand\rotatebox[2]{#2}%
  \newcommand*\fsize{\dimexpr\f@size pt\relax}%
  \newcommand*\lineheight[1]{\fontsize{\fsize}{#1\fsize}\selectfont}%
  \ifx\svgwidth\undefined%
    \setlength{\unitlength}{533.97620441bp}%
    \ifx\svgscale\undefined%
      \relax%
    \else%
      \setlength{\unitlength}{\unitlength * \real{\svgscale}}%
    \fi%
  \else%
    \setlength{\unitlength}{\svgwidth}%
  \fi%
  \global\let\svgwidth\undefined%
  \global\let\svgscale\undefined%
  \makeatother%
  \begin{picture}(1,0.34960612)%
    \lineheight{1}%
    \setlength\tabcolsep{0pt}%
    \put(0,0){\includegraphics[width=\unitlength,page=1]{embeds4.pdf}}%
    \put(0.79282024,0.07545137){\color[rgb]{0,0,0}\makebox(0,0)[lt]{\lineheight{1.25}\smash{\begin{tabular}[t]{l}$c$\end{tabular}}}}%
    \put(0.6082544,0.2472485){\color[rgb]{0,0,0}\makebox(0,0)[lt]{\lineheight{1.25}\smash{\begin{tabular}[t]{l}$-1$\end{tabular}}}}%
    \put(0.70227847,0.0150903){\color[rgb]{0,0,0}\makebox(0,0)[lt]{\lineheight{1.25}\smash{\begin{tabular}[t]{l}$b$\end{tabular}}}}%
    \put(0.55369726,0.07429061){\color[rgb]{0,0,0}\makebox(0,0)[lt]{\lineheight{1.25}\smash{\begin{tabular}[t]{l}$a$\end{tabular}}}}%
    \put(0.71504717,0.2576956){\color[rgb]{0,0,0}\makebox(0,0)[lt]{\lineheight{1.25}\smash{\begin{tabular}[t]{l}$-2k-2$\end{tabular}}}}%
    \put(0.66861556,0.31109195){\color[rgb]{0,0,0}\makebox(0,0)[lt]{\lineheight{1.25}\smash{\begin{tabular}[t]{l}$1$\end{tabular}}}}%
    \put(0.0678134,0.08125542){\color[rgb]{0,0,0}\makebox(0,0)[lt]{\lineheight{1.25}\smash{\begin{tabular}[t]{l}$0$\end{tabular}}}}%
    \put(0.177601,0.19385216){\color[rgb]{0,0,0}\makebox(0,0)[lt]{\lineheight{1.25}\smash{\begin{tabular}[t]{l}$2k+1$\end{tabular}}}}%
  \end{picture}%
\endgroup%

%% file: ms.pdf_tex
\begingroup%
  \makeatletter%
  \providecommand\color[2][]{%
    \errmessage{(Inkscape) Color is used for the text in Inkscape, but the package 'color.sty' is not loaded}%
    \renewcommand\color[2][]{}%
  }%
  \providecommand\transparent[1]{%
    \errmessage{(Inkscape) Transparency is used (non-zero) for the text in Inkscape, but the package 'transparent.sty' is not loaded}%
    \renewcommand\transparent[1]{}%
  }%
  \providecommand\rotatebox[2]{#2}%
  \newcommand*\fsize{\dimexpr\f@size pt\relax}%
  \newcommand*\lineheight[1]{\fontsize{\fsize}{#1\fsize}\selectfont}%
  \ifx\svgwidth\undefined%
    \setlength{\unitlength}{385.44610398bp}%
    \ifx\svgscale\undefined%
      \relax%
    \else%
      \setlength{\unitlength}{\unitlength * \real{\svgscale}}%
    \fi%
  \else%
    \setlength{\unitlength}{\svgwidth}%
  \fi%
  \global\let\svgwidth\undefined%
  \global\let\svgscale\undefined%
  \makeatother%
  \begin{picture}(1,0.69124701)%
    \lineheight{1}%
    \setlength\tabcolsep{0pt}%
    \put(0,0){\includegraphics[width=\unitlength,page=1]{ms.pdf}}%
    \put(-0.00142351,0.44867141){\color[rgb]{0,0,0}\makebox(0,0)[lt]{\lineheight{1.25}\smash{\begin{tabular}[t]{l}$0$\end{tabular}}}}%
    \put(0.05767244,0.46302838){\color[rgb]{0,0,0}\makebox(0,0)[lt]{\lineheight{1.25}\smash{\begin{tabular}[t]{l}$\frac{1}{4}$\end{tabular}}}}%
    \put(0.11889211,0.46275448){\color[rgb]{0,0,0}\makebox(0,0)[lt]{\lineheight{1.25}\smash{\begin{tabular}[t]{l}$\frac{1}{2}$\end{tabular}}}}%
    \put(0.25998408,0.44783513){\color[rgb]{0,0,0}\makebox(0,0)[lt]{\lineheight{1.25}\smash{\begin{tabular}[t]{l}$1$\end{tabular}}}}%
    \put(0.33533433,0.00734234){\color[rgb]{0,0,0}\makebox(0,0)[lt]{\lineheight{1.25}\smash{\begin{tabular}[t]{l}$\phi_1$\end{tabular}}}}%
    \put(0.91162016,0.01205541){\color[rgb]{0,0,0}\makebox(0,0)[lt]{\lineheight{1.25}\smash{\begin{tabular}[t]{l}$\phi_2$\end{tabular}}}}%
    \put(0.62233097,0.65646628){\color[rgb]{0,0,0}\makebox(0,0)[lt]{\lineheight{1.25}\smash{\begin{tabular}[t]{l}$0$\end{tabular}}}}%
    \put(0.89581822,0.65993907){\color[rgb]{0,0,0}\makebox(0,0)[lt]{\lineheight{1.25}\smash{\begin{tabular}[t]{l}$1$\end{tabular}}}}%
    \put(0.75378757,0.67461531){\color[rgb]{0,0,0}\makebox(0,0)[lt]{\lineheight{1.25}\smash{\begin{tabular}[t]{l}$\frac{1}{2}$\end{tabular}}}}%
    \put(0.82931544,0.67052935){\color[rgb]{0,0,0}\makebox(0,0)[lt]{\lineheight{1.25}\smash{\begin{tabular}[t]{l}$\frac{3}{4}$\end{tabular}}}}%
    \put(0,0){\includegraphics[width=\unitlength,page=2]{ms.pdf}}%
    \put(0.82535323,0.33605139){\color[rgb]{0,0,0}\makebox(0,0)[lt]{\lineheight{1.25}\smash{\begin{tabular}[t]{l}$\mathbb{D}^1$\end{tabular}}}}%
    \put(0.77279225,0.40372704){\color[rgb]{0,0,0}\makebox(0,0)[lt]{\lineheight{1.25}\smash{\begin{tabular}[t]{l}$\mathbb{D}^n$\end{tabular}}}}%
    \put(0,0){\includegraphics[width=\unitlength,page=3]{ms.pdf}}%
    \put(0.23905338,0.30855214){\color[rgb]{0,0,0}\makebox(0,0)[lt]{\lineheight{1.25}\smash{\begin{tabular}[t]{l}$\mathbb{D}^1$\end{tabular}}}}%
    \put(0.28497149,0.19751906){\color[rgb]{0,0,0}\makebox(0,0)[lt]{\lineheight{1.25}\smash{\begin{tabular}[t]{l}$\mathbb{D}^n$\end{tabular}}}}%
    \put(0,0){\includegraphics[width=\unitlength,page=4]{ms.pdf}}%
  \end{picture}%
\endgroup%